\documentclass[12pt]{article}

\title{Quantitative homogenization with relatively soft inclusions and interior estimates}

\date{}
\author{B. Chase Russell\footnote{The author is supported in part by NSF grant DMS-1600520.  The author is also grateful for the valuable comments, suggestions, and advice of Z. Shen.}}

\usepackage{amsmath,amsthm,amssymb}

\renewcommand{\S}[3]{\sigma_{#1}({#2})={#3}}

\newcommand{\dashint}{-\!\!\!\!\!\!\displaystyle\int}

\newcommand{\dint}{\displaystyle\int}
\newcommand{\R}{\mathbb{R}}

\newcommand{\X}{\mathbb{X}}
\newcommand{\K}{\mathcal{K}}

\newcommand{\dd}{\partial}

\newcommand{\smooth}[2]{K_{#2}{#1}}
\newcommand{\smoothone}[2]{K_{#2}{\left(#1\right)}}
\newcommand{\smoothtwo}[2]{K^2_{#2}{\left(#1\right)}}
\newcommand{\e}{\epsilon}

\renewcommand{\a}{\alpha}
\renewcommand{\b}{\beta}
\renewcommand{\d}{\delta}
\newcommand{\g}{\gamma}
\newcommand{\z}{\zeta}

\newcommand{\x}{\xi}

\renewcommand{\epsilon}{\varepsilon}

\newcommand{\D}{\mathcal{D}}
\renewcommand{\S}{\mathcal{S}}

\newcommand{\one}{\textbf{1}}
\newcommand{\lrnoe}[1]{k_{\d^{#1}}}
\newcommand{\lr}[1]{k^\e_{\d^{#1}}}

\newcommand{\G}{\Gamma}

\theoremstyle{plain}
\newtheorem{lemm}{Lemma}
\newtheorem{thmm}[lemm]{Theorem}
\newtheorem{corr}[lemm]{Corollary}

\numberwithin{equation}{section}
\numberwithin{lemm}{section}

\allowdisplaybreaks

\begin{document}

\maketitle

\abstract{We establish large-scale interior Lipschitz estimates for solutions to systems of linear elasticity with rapidly oscillating periodic coefficients and Dirichlet boundary conditions in domains with periodically placed inclusions of size $\mathcal{O}(\e)$ and magnitude $\d$ by establishing $H^1$-convergence rates for such solutions.  The interior estimates at the macroscopic scale are derived directly without the use of compactness via a Campanato-type scheme presented by S. Armstrong and C.K. Smart and that was adapted for uniformly elliptic equations in by Armstrong and Z. Shen.
}

\vspace{5mm}



%
%
%
%
%
\section{Introduction}\label{section1}
%
%
%
%
%

The purpose of this paper is to establish large-scale interior Lipschitz estimates for solutions to systems of linear elasticity with $\e$-periodic coefficients in domains with periodically placed inclusions of size $\mathcal{O}(\e)$ and magnitude $\d$ and to establish $H^1$-convergence rates in periodic homogenization.  To be precise, let $\omega\subseteq\R^d$ be an unbounded domain with 1-periodic structure, i.e., if $\one_+$ denotes the characteristic function of $\omega$, then $\one_+$ is a 1-periodic function in the sense that
\begin{equation}\label{zero}
\one_+(y)=\one_+(z+y)\,\,\,\text{ for }y\in\R^d,\,z\in\mathbb{Z}^d.
\end{equation}
Let $\one_-$ denote the characteristic function of $\R^d\backslash\omega$, and note it also satisfies~\eqref{zero}.  For $\e>0$, $0\leq \d\leq 1$, we consider the operator
\begin{equation}\label{fortythree}
\mathcal{L}_{\e,\d}=-\text{div}(\lr{} A^\e(x)\nabla)=-\dfrac{\dd}{\dd x_i}\left(\lr{}a_{ij}^{\a\b}\left(\dfrac{x}{\e}\right)\dfrac{\dd}{\dd x_j}\right),
\end{equation}
for $x\in\R^d$, where $A^\e(x)=A(x/\e)$, $A(y)=\{a_{ij}^{\a\b}(y)\}_{1\leq i,j,\a,\b\leq d}$ for $y\in\R^d$, $d\geq 2$, $\lr{}=\lrnoe{}(\cdot/\e)$, and
\[
\lrnoe{}(y)=\one_+(y)+\d\one_-(y).
\]
The specific case $\d=0$ is discussed in~\cite{bcr17}.  Naturally, $\lrnoe{}$ is 1-periodic.  We assume the coefficient matrix $A(y)$ is real, measurable, and satisfies the elasticity conditions
\begin{align}
& a_{ij}^{\a\b}(y)=a_{ji}^{\b\a}(y)=a_{\a j}^{i\b}(y), \label{one}\\
& \kappa_1|\x|^2\leq a_{ij}^{\a\b}(y)\x_i^\a\x_j^\b\leq\kappa_2|\x|^2,\label{two}
\end{align}
for a.e $y\in\R^d$ and any symmetric matrix $\x=\{\x_i^\a\}_{1\leq i,\a\leq d}$, where $\kappa_1,\kappa_2>0$.  We also assume $A$ is 1-periodic in the sense of~\eqref{zero}, i.e., 
\begin{equation}\label{four}
A(y)=A(y+z)\,\,\,\text{ for }y\in\R^d,\,z\in\mathbb{Z}^d.
\end{equation}
The coefficient matrix of the systems of linear elasticity describes the relation between the stress and strain a material experiences during relatively small elastic deformations.  Consequently, the elasticity conditions~\eqref{one},~\eqref{two}, and $\d$ should be regarded as physical parameters of the system, whereas $\e$ and~\eqref{four} are clearly geometric characteristics of the system.  

Let $\Omega$ be a bounded domain.  In this paper, we consider the Dirichlet boundary value problem given by
\begin{equation}\label{three}
\begin{cases}
\mathcal{L}_{\e,\d}(u_{\e,\d})=0\,\,\,\text{ in }\Omega, \\
u_{\e,\d}=f\,\,\,\text{ on }\dd\Omega.
\end{cases}
\end{equation}
We say $u_{\e,\d}$ is a weak solution to~\eqref{three} provided
\begin{equation}\label{fortyfour}
\dint_{\Omega}\lr{}a_{ij}^{\a\b\e}\dfrac{\dd u_{\e,\d}^\b}{\dd x_j}\dfrac{\dd w^\a}{\dd x_i}=0\,\,\,\text{ for any }w=\{w^\a\}_\a\in H^1_0(\Omega;\R^d)
\end{equation}
and $u_{\e,\d}-f_{}\in H_0^1(\Omega;\R^d)$.  Note when $\d=0$, Neumann boundary conditions on the perforations are implied.  The boundary value problem~\eqref{three} models relatively small elastic deformations of composite materials reinforced with soft inclusions and subject to zero external body forces~\cite{cioranescu,yellowbook,book2}.  In particular, soft inclusions are comparatively ``weaker'' than the cementing matrix $\omega$, but their embedding can be otherwise advantageous  For example, a material's compressive strength can be indirectly proportional with the increasing volume of soft inclusions but the thermal inertia and energy efficiency may be directly proportional~\cite{falzone}.

For each $\d\in(0,1]$, the existence and uniqueness of a weak solution $u_{\e,\d}\in H^1(\Omega;\R^d)$ to~\eqref{three} for $f_{}\in H^{1/2}(\dd\Omega;\R^d)$ follows easily from the Lax-Milgram theorem and Korn's inequality.  For $\d=0$, the existence and unqieness follows from Lax-Milgram and Korn's inequality for perforated domains~\cite{book1,book2}.  It should be noted that the solution $u_{\e,\d}$ is not bounded uniformly in $H^1(\Omega;\R^d)$.  Indeed, if $\mathcal{L}_{\e,\d}(u_{\e,\d})=0$ in $\Omega$ and $u_{\e,\d}=f$ on $\dd\Omega$, then one may deduce by energy estimates
\[
\|\lr{}u_{\e,\d}\|_{L^2(\Omega)}+\|\lr{}\nabla u_{\e,\d}\|_{L^2(\Omega)}\leq C\|f\|_{H^{1/2}(\dd\Omega)},
\]
where $C$ depends on $\kappa_1$, $\kappa_2$.

One of the main results of this paper is the following theorem.  We emphasize that no smoothness assumptions are required on the coefficients $A$, only the elasticity conditions~\eqref{one},~\eqref{two}, and the periodicity condition~\eqref{four}.  

\begin{thmm}\label{five}
Suppose $A$ satisfies~\eqref{one},~\eqref{two}, and~\eqref{four}.  Let $u_{\e,\d}$ denote a weak solution to $\mathcal{L}_{\e,\d}(u_{\e,\d})=0$ in $B(x_0,R)$ for some $x_0\in\R^d$ and $R>0$.  For $\e\leq r\leq R$, there exists a constant $C$ depending on $d$, $\omega$, $\kappa_1$, and $\kappa_2$ such that
\[
\left(\dashint_{B(x_0,r)}|\lr{}\nabla u_{\e,\d}|^2\right)^{1/2}\leq C\left(\dashint_{B(x_0,R)}|\lr{}\nabla u_{\e,\d}|^2\right)^{1/2}.
\]
for $0\leq \d\leq1$.
\end{thmm}

The scale-invariant estimate in Theorem~\ref{five} should be regarded as a Lipschitz estimate at the large scale, e.g., $1\leq r/\e$, and it is proved in Section~\ref{section4}.  Indeed, if Theorem~\ref{five} were to hold also for $0<r<\e$, then by letting $r\to 0$ we would have
\[
|\lr{}\nabla u_{\e,\d}(x_0)|\leq C\left(\dashint_{B(x_0,R)}|\lr{}\nabla u_{\e,\d}|^2\right)^{1/2}
\]
for all $x_0$ in some compact subset of $\Omega$.  In particular, we would have a Lipschitz estimate indepedent of $\d$ for $u_{\e,\d}$ in the connected substrate $\omega$ and a Lipschitz estimate for $u_{\e,\d}$ in the inclusions $\R^d\backslash\omega$ with explicit knowledge of the effect of the parameter $\d$.  Unfortunately,~\eqref{eight} does not hold without more assumptions on the smoothness of the coefficients $A$ and the domain $\omega$.  That is, the periodicity assumptions~\eqref{zero},~\eqref{four} and elasticity conditions~\eqref{one},~\eqref{two} alone contribute to the large-scale average behavior of the solution.

Under additional assumptions that $A$ is H\"older continuous and the domain $\omega$ has a sufficiently regular boundary, an interior Lipschitz estimate at the microscopic scale for solutions to~\eqref{three} follows from local $C^{1,\a}$-estimates for the operator $\mathcal{L}_{1,\d}$.  This follows from a layer potential argument of Escaurazia, Fabes, and Verchota where nontangential estimates were obtained for single equation interface problems~\cite{fabes}.  Yeh modified this same argument to obtain local $W^{1,p}$-estimates and H\"older estimates for~\eqref{three} in the case of single equations with diagonal coefficients~\cite{yeh10,yeh16}.  The necessary modifications for out setting is discussed in Appendix~\ref{section5}.  

Nevertheless, if $A$ is $\a$-H\"older continuous, i.e., there exists a $\a\in (0,1)$ with
\begin{equation}\label{six}
|A(x)-A(y)|\leq C|x-y|^\a\,\,\,\text{ for }x,y\in\R^d
\end{equation}
for some constant $C$ uniform in $x$ and $y$, then the following corollary holds.

\begin{corr}\label{threethousand}
Suppose $A$ satisfies~\eqref{one},~\eqref{two},~\eqref{four}, and~\eqref{six} for some $\a\in(0,1)$.  Suppose $\omega$ is an unbounded $C^{1,\a}$ domain.  Let $u_{\e,\d}$ denote a weak solution to $\mathcal{L}_{\e,\d}(u_{\e,\d})=0$ in $B(x_0,R)$ for some $x_0\in\R^d$ and $R>0$.  Then for $0\leq\d\leq 1$,
\begin{equation}\label{eight}
\|\lr{}\nabla u_{\e,\d}\|_{L^\infty(B(x_0,R/3))}\leq C\left(\dashint_{B(x_0,R)}|\lr{}\nabla u_{\e,\d}|^2\right)^{1/2}
\end{equation}
some constant $C$ independent of $\e$ and $\d$.
\end{corr}


Interior Lipschitz estimates for the case $\d=1$ were first obtained \textit{indirectly} through the method of compactness by Avellaneda and Lin~\cite{avellaneda}.  The celebrated method of compactness has been applied in other settings~\cite{shu,yeh10,yeh16}.  For example, uniform H\"older estimates for a single elliptic equation with diagonal coefficients in the case $\d=\e$ were obtained indirectly by Yeh with this method~\cite{yeh10}.  The method of compactness is esentially ``proof by contradiction'' and relies on qualitative convergence, which for~\eqref{three} can be ambiguous and complicated.

Interior Lipschitz estimates for the case $\d=1$ were obtained \textit{directly} by Shen~\cite{shen} through a a general scheme developed by Armstrong and Smart~\cite{smart} for establishing large-scale Lipschitz estimates for local minimizers of convex integral functionals arising in homogenzation.  The method was adapted for divergence form elliptic equations with almost-periodic coefficients by Armstrong and Shen~\cite{armstrong2}.  The same estimates were directly proved for the case $\d=0$ by the author of this paper using the general scheme~\cite{bcr17}.  Essentially, in this paper we establish sub-optimal quantitative convergence rates for solutions to~\eqref{three} and use the same scheme.

Hueristically, the scheme is a Campanato-type iteration verifying that on mesoscopic scales the solution $u_{\e,\d}$ is ``flatter.''  If $P_1$ denotes the space of affine functions in $\R^d$ and $H_{\e,\d}(r)$ defined by
\[
H_{\e,\d}(r)=\dfrac{1}{r}\left(\underset{p\in P_1}{\inf}\dashint_{B(r)}|\lr{}(u_{\e,\d}-p)|^2\right)^{1/2}
\]
quantifies a weighted $L^2$-``flatness'' of the solution in some ball $B(r)$ with radius $r$, then we show there exists a $\theta\in (0,1)$ such that
\begin{equation}\label{onethousand}
H_{\e,\d}(\theta r)\leq CH_{\e,\d}(r)+\text{error},
\end{equation}
where the ``error'' term is controllable whenever $\e\leq r$ and the constant $0\leq C<1$ indicates an improvement in ``flatness.''  Indeed,~\eqref{onethousand} follows from the fact that $u_{\e,\d}$---at least in the connected substrate---can be well-approximated in $L^2$ by a solution to a constant coefficient system.  It is known from classical $C^2$ estimates that solutions to constant coefficient systems satisfy~\eqref{onethousand} with no error.  In contrast to compactness methods, showing~\eqref{onethousand} relies on tractable $L^2$-convergence rates of $u_{\e,\d}$, which we will see follows from new results regarding quantitative homogenization in $H^1$.  These sub-optimal $H^1$-convergence rates are stated in Theorem~\ref{twelve} and proved in Section~\ref{section3}.

%

For fixed $\d\geq 0$, the estimate
\begin{equation}\label{twentythree}
\|u_{\e,\d}-u_{0,\d}-\e\chi_{\d}^\e\smoothtwo{(\nabla u_{0,\d})\eta_\e}{\e}\|_{H^1(\Omega)}\leq C\e^{1/2}\|u_{0,\d}\|_{H^1(\dd\Omega)}.
\end{equation}
is known, where $u_{0,\d}\in H^1(\Omega;\R^d)$ denotes the weak solution of the boundary value problem for the homogenized system corresponding to~\eqref{three}, $\chi_\d=\{\chi_{j,\d}^\b\}_{1\leq j,\b\leq d}\in H^1_{\text{per}}(\R^d;\R^d)$ denotes the matrix of correctors associated with the coefficients $\lrnoe{}A$ (see~\eqref{nineteen}), $K_\e$ denotes the smoothing operator at scale $\e$ defined by~\eqref{ten}, and $\eta_\e\in C_0^\infty(\Omega)$ be the cut-off function defined by~\eqref{873}.  However, the explicit dependence of $C$ on the parameter $\d$ is not known.  The estimate~\eqref{twentythree} was proved by the author of this paper in~\cite{bcr17} when $\d=0$.  For $\d=1$, the estimate was proved by Shen in~\cite{shen}.  The following theorem is therefore also a main result of this paper, as it holds for any $0\leq \d\leq 1$ and the constant $C$ is completely independent of the parameter $\d$.

\begin{thmm}\label{twelve}
Let $\Omega$ be a bounded Lipschitz domain and $\omega$ be an unbounded Lipschitz domain with 1-periodic structure.  Suppose $A$ is real, measurable, and satisfies~\eqref{one},~\eqref{two}, and~\eqref{four}.  Let $u_{\e,\d}$ denote a weak solution to~\eqref{three} for $0\leq \d\leq 1$.  There exists a constant $C$ depending on $\kappa_1$, $\kappa_2$, $d$, $\Omega$, and $\omega$ and a $\mu>0$ depending on $\kappa_1$, $\kappa_2$, $d$, and $\Omega$ such that
\begin{align}\label{1596478}
&\|\lr{}r_{\e,\d}\|_{L^2(\Omega)}+\|\lr{}\nabla r_{\e,\d}\|_{L^2(\Omega)}\leq C\e^{\mu}\|f_{}\|_{H^1(\dd\Omega)},
\end{align}
where
\begin{equation}\label{nine}
r_{\e,\d}=u_{\e,\d}-u_{0,\d}-\e\chi_{\d}^\e\smoothtwo{(\nabla u_{0,\d})\eta_\e}{\e}.
\end{equation}
and $u_{0,\d}=f$ on $\dd\Omega$.
\end{thmm}

Theorem~\ref{twelve} is particularly new for small yet positive $\d$.  Indeed, if $0<\d_0\leq \d\leq 1$, then estimate~\eqref{1596478} follows from work in~\cite{shen} where the constant $C$ depends somehow on $\d_0$.  With regards to the regularity estimates, the theorem essentially establishes the estimate
\begin{equation}\label{twothousand}
\left(\dashint_{B(r)}|\lr{}(u_{\e,\d}-v)|^2\right)^{1/2}\lesssim\left(\dfrac{\e}{r}\right)^{\mu}
\end{equation}
for some $v$ satisfying a constant coefficient system.  The established $C^2$ estimates for $v$ are used to give~\eqref{onethousand}.  The ``error'' term of~\eqref{onethousand} is on the order of the RHS of~\eqref{twothousand}, and so to carry out the scheme it is important that $\mu>0$.  The typical iteration of~\eqref{onethousand} depends on the smallness of the RHS of~\eqref{twothousand} and is written in full detail in~\cite{armstrong2,smart}.  We use a generalization of the iteration process provided by Shen~\cite{shen} (see Lemma~\ref{sixtythree}).

The paper is structured in the following way.  In Section~\ref{section2}, we establish more notation and recall various preliminary results from other works.  The convergence rate presented in Theorem~\ref{twelve} is proved in Section~\ref{section3}.  In Section~\ref{section4}, we prove the interior Lipschitz estimates at the macroscopic scale, i.e., Theorem~\ref{five}.  In Appendix~\ref{section5}, we argue the local interior Lipschitz estimates at the microscopic scale by applying the argument of~\cite{fabes}.  It should be noted throughout that $C$ is a harmless constant that may be change from line to line.  At no point does $C$ depend on $\e$ or $\d$.

%
%
%
%
%
\section{Preliminaries}\label{section2}
%
%
%
%
%

Fix $\varphi\in C_0^\infty(B(0,1))$ with $\varphi\geq 0$ and $\int_{\R^d}\varphi=1$.  Define
\begin{equation}\label{ten}
\smooth{g}{\e}(x)=\dint_{\R^d}g(x-y)\varphi_\e(y)\,dy,\,\,\,g\in L^2(\R^d),
\end{equation}
where $\varphi_\e(y)=\e^{-d}\varphi(y/\e)$.  Note $K_\e$ is a continuous map from $L^2(\R^d)$ to $L^2(\R^d)$.  A proof for the following two lemmas is readily available in~\cite{shen}, and so we do not present either here.  For any function $g$, set $g^\e(\cdot)=g(\cdot/\e)$.

\begin{lemm}\label{thirtysix}
Let $g\in H^1(\R^d)$.  Then
\[
\|g-\smoothone{g}{\e}\|_{L^2(\R^d)}\leq C\e\|\nabla g\|_{L^2(\R^d)},
\]
where $C$ depends only on $d$.
\end{lemm}

\begin{lemm}\label{twentyseven}
Let $h\in L^2_{\text{loc}}(\R^d)$ be a 1-periodic function.  Then for any $g\in L^2(\R^d)$, 
\[
\|h^\e\smoothone{g}{\e}\|_{L^2(\R^d)}\leq C\|h\|_{L^2(Q)}\|g\|_{L^2(\R^d)},
\]
where $Q=[0,1)^d$ and $C$ depends on $d$ and $\Omega$.
\end{lemm}

A proof of Lemma~\ref{thirteen} can be found in~\cite{book2}.

\begin{lemm}\label{thirteen}
Let $\Omega\subset\R^d$ be a bounded Lipschitz domain.  For any $g\in H^1(\Omega)$,
\[
\|g\|_{L^2(\mathcal{O}_r)}\leq Cr^{1/2}\|g\|_{H^1(\Omega)},
\]
where $\mathcal{O}_r=\{x\in\Omega\,:\,\text{dist}(x,\dd\Omega)<r\}$ and $C$ depends only on $d$.
\end{lemm}

A proof of Lemma~\ref{fourteen} can be found in~\cite{yellowbook}.

\begin{lemm}\label{fourteen}
Suppose $B=\{b_{ij}^{\a\b}\}_{1\leq i,j,\a,\b\leq d}$ is 1-periodic and satisfies $b_{ij}^{\a\b}\in L_{\text{loc}}^2(\R^d)$ with
\[
\dfrac{\dd}{\dd y_i}b_{ij}^{\a\b}=0,\,\,\,\text{ and }\,\,\,\dint_Q b_{ij}^{\a\b}=0.
\]
There exists $\pi=\{\pi_{kij}^{\a\b}\}_{1\leq i,j,k,\a,\b\leq d}$ with $\pi_{kij}^{\a\b}\in H^1_{\text{loc}}(\R^d)$ that is 1-periodic and satisfies
\[
\dfrac{\dd}{\dd y_k}\pi_{kij}^{\a\b}=b_{ij}^{\a\b}\,\,\,\text{ and }\,\,\,\pi_{kij}^{\a\b}=-\pi_{ikj}^{\a\b}.
\]
\end{lemm}

If $\d>0$, it can be shown that the weak solution to~\eqref{three} converges weakly in $H^1(\Omega;\R^d)$ and consequently strongly in $L^2(\Omega;\R^d)$ as $\e\to 0$ to some $u_{0,\d}$, which is a solution of a constant-coefficient equation in the domain $\Omega$ (see~\cite{monograph,yellowbook,book2} and references therein).  Indeed, the following theorem is well-known.

\begin{thmm}\label{fifteen}
Suppose $\Omega$ is a bounded Lipschitz domain and that $A$ satisfies ~\eqref{one},~\eqref{two}, and~\eqref{four}.  Let $u_{\e,\d}$ satisfy $\mathcal{L}_{\e,\d}(u_{\e,\d})=0$ in $\Omega$, and $u_{\e,\d}=f$ on $\dd\Omega$ for some fixed $\d>0$.  Then there exists a $u_{0,\d}\in H^1(\Omega;\R^d)$ such that
\[
u_{\e,\d}\rightharpoonup u_{0,\d}\,\,\,\text{ weakly in }H^1(\Omega;\R^d).
\]
Consequently, $u_{\e,\d}\to u_{0,\d}$ strongly in $L^2(\Omega;\R^d)$.
\end{thmm}

For a proof of the previous theorem, see~\cite[Section 10.3]{book1} and notice $k_\d^\e A^\e$ is uniformly elliptic in $\R^d$ for $\d>0$.  The function $u_{0,\d}$ is called the homogenized solution and the boundary value problem it solves is the homogenized system corresponding to~\eqref{three}.

Theorem~\ref{twentyfive} is a typical result in the study of periodically perforated domains, i.e., the case when $\d=0$.  For a proof of the following theorem, consult the work of Acerbi, Piat, Dal Maso, and Percivale~\cite{acerbi}.  Let $\Omega_\e=\Omega\cap\e\omega$ and for $p\in (1,\infty)$ let $W^{1,p}(\Omega_\e,\Gamma_\e;\R^d)$ denote the closure in $W^{1,p}(\Omega_\e;\R^d)$ of $C^\infty(\R^d;\R^d)$ function vanishing on $\Gamma_\e:=\dd\Omega_\e\cap\dd\Omega$.

\begin{thmm}\label{twentyfive}
Let $\Omega$ and $\Omega_0$ be bounded Lipschitz domains with $\overline{\Omega}\subset\Omega_0$ and $\text{dist}(\dd\Omega_0,\Omega)>1$.  Let $p\in (1,\infty)$.  For $\e$ small enough, there exists a linear extension operator $P_\e: W^{1,p}(\Omega_\e,\Gamma_\e;\R^d)\to W_0^{1,p}(\Omega_0;\R^d)$ such that 
\begin{align}
&P_\e w=w\,\,\,\text{a.e. in }\Omega_\e \\
&\|P_\e w\|_{L^p(\Omega_0)}\leq C_1\|w\|_{L^p(\Omega_\e)}, \label{}\\
&\|\nabla P_\e w\|_{L^p(\Omega_0)}\leq C_2\|\nabla w\|_{L^p(\Omega_\e)}, \label{fortynine}
\end{align}
for some constants $C_1$, $C_2$ depending on $\Omega$, $\omega$, $d$, and $p$.
\end{thmm}

If $\d=0$, then it is difficult to qualitatively discuss the convergence of $u_{\e,0}$, although in this case it is discussed in~\cite{siam,monograph,book1,cioranescu} and many others.  Quantitatively, however, we have the estimate~\eqref{twentythree}.  A stronger estimate for this case is proved in~\cite{bcr17}.  The homogenized system of elasticity corresponding to~\eqref{three} in the case $\d=0$ and of which $u_{0,0}$ is a solution is given by
\begin{equation}\label{sixteen}
\begin{cases}
\mathcal{L}_{0,0}\left({u_{0,0}}\right)={0}\,\,\,\text{ in }\Omega \\
u_{0,0}=f_{}\,\,\,\text{ on }\dd\Omega,
\end{cases}
\end{equation}
where $\mathcal{L}_{0,0}=-\text{div}(\widehat{A}_0\nabla)$, $\widehat{A}_0=\{\widehat{a}_{ij,0}^{\a\b}\}_{1\leq i,j,\a,\b\leq d}$ denotes a constant matrix given by
\begin{equation}\label{seventeen}
\widehat{a}_{ij,0}^{\a\b}=\dashint_{Q}k_0a_{ik}^{\a\g}\dfrac{\dd\X_{j,0}^{\g\b}}{\dd y_k},
\end{equation}
and $\X_{j,0}^\b=\{\X_{j,0}^{\g\b}\}_{1\leq\g\leq d}$ denotes the solution to the following variational problem
\begin{equation}\label{eighteen}
\begin{cases}
\dint_{Q}k_0a_{ik}^{\a\g}\dfrac{\dd\X_{j,0}^{\g\b}}{\dd y_k}\dfrac{\dd \phi^\a}{\dd y_i}\,dy=0,\,\,\,\text{ for any }\phi=\{\phi^\a\}_\a\in H^1_{\text{per}}(Q;\R^d) \\
\chi_{j,0}^\b:=\X_{j,0}^\b-y_je^\b\text{ is 1-periodic},\,\,\,\dint_{Q}k_0\chi_{j,0}^\b=0,
\end{cases}
\end{equation}
where $e^\b\in\R^d$ has a 1 in the $\b$th position and 0 in the remaining positions and $H^1_{\text{per}}(Q;\R^d)$ denotes the closure of $C^\infty_{\text{per}}(\R^d;\R^d)$ functions in the $H^1(Q;\R^d)$ norm.  For details on the existence of solutions to~\eqref{eighteen}, see~\cite{book2}.  The functions $\chi^{\b}_{j,0}$ are referred to as the first-order correctors for the system~\eqref{three} with $\d=0$.  The coefficients $\widehat{A}_0$ are known to be uniformly elliptic.  Indeed, we have the following lemma.  For a proof, see~\cite{book2}.

\begin{lemm}\label{twentysix}
Suppose $A$ satisfies~\eqref{one},~\eqref{two}, and~\eqref{four}.  If $\X_{j,0}^\b=\{\X_{j,0}^{\g\b}\}$ are defined by~\eqref{eighteen}, then $\widehat{A}_0=\{a_{ij,0}^{\a\b}\}$ defined by~\eqref{seventeen} satisfies
\[
\widehat{a}_{ij,0}^{\a\b}=\widehat{a}_{\a j,0}^{i\b}=\widehat{a}_{ji,0}^{\a\b}
\]
for $1\leq i,j,\a,\b\leq d$ and
\[
\widehat{\kappa}_1|\x|^2\leq \widehat{a}_{ij,0}^{\a\b}\x_i^\a\x_j^\b\leq \widehat{\kappa}_2|\x|^2
\]
for some $\widehat{\kappa}_1,\widehat{\kappa}_2>0$ and any symmetric matrix $\x=\{\x_i^\a\}_{i,\a}$.
\end{lemm}

For $\d\geq 0$, let $\chi_{j,\d}^\b=\{\chi_{j,\d}^{\g\b}\}_{1\leq\g\leq d}$ denote the solution to the following variational problem
\begin{equation}\label{nineteen}
\begin{cases}
\dint_{Q}\lrnoe{}a_{ik}^{\a\g}\dfrac{\dd\X_{j,\d}^{\g\b}}{\dd y_k}\dfrac{\dd \phi^\a}{\dd y_i}\,dy=0,\,\,\,\text{ for any }\phi\in H^1_{\text{per}}(Q;\R^d) \\
\chi_{j,\d}^\b:=\X_{j,\d}^\b-y_je^\b\text{ is 1-periodic},\,\,\,\dint_{Q}\chi_{j,\d}^\b=0,
\end{cases}
\end{equation}
which coincides with~\eqref{eighteen} if $\d=0$.  To show the existence and uniqueness of the solutions $\chi_{j,\d}^\b$, apply the Lax-Milgram theorem to the space $H^1_{\text{per}}(Q;\R^d)$.  As a consequence, with the appropriate choice of test functions, one can obtain the bound
\[
\|\lrnoe{}\chi_{j,\d}^\b\|_{L^2(Q)}+\|\lrnoe{}\nabla \chi_{j,\d}^\b\|_{L^2(Q)}\leq C
\]
for some constant $C$ depending on $\kappa_1$, $\kappa_2$, and $\omega$.

Define the constant matrix $\widehat{A}_\d=\{a_{ij,\d}^{\a\b}\}$ by
\begin{equation}\label{twenty}
\widehat{a}_{ij,\d}^{\a\b}=\dashint_Q \lrnoe{}a_{ik}^{\a\g}\dfrac{\dd\X_{j,\d}^{\g\b}}{\dd x_k}\,dy,
\end{equation}
where $\X_{j,\d}^{\b}$ is defined in~\eqref{nineteen}.  The constant matrix $\widehat{A}_\d$ is uniformly elliptic uniformly in $\d$.  For details, see Section~\ref{section3}.  Let $u_{0,\d}$ denote a solution to the homogenized boundary value problem corresponding to~\eqref{three} with $\d\geq 0$, i.e., $u_{0,\d}$ satisfies
\begin{equation}\label{874}
\begin{cases}
\mathcal{L}_{0,\d}(u_{0,\d})=0\,\,\,\text{ in }\Omega \\
u_{0,\d}=f\,\,\,\text{ on }\dd\Omega,
\end{cases}
\end{equation}
where $\mathcal{L}_{0,\d}=-\text{div}(\widehat{A}_\d\nabla )$ and $\widehat{A}_\d$ is defined by~\eqref{twenty}.

Throughout, it is assumed that any two connected components of $\R^d\backslash\omega$ are separated by some positive distance.  Specifically, if $\R^d\backslash\omega=\cup_{k=1}^\infty H_k$, where $H_k$ is simply connected and bounded for each $k$, then there exists a constant $\mathfrak{g}^\omega$ so that 
\begin{equation}\label{nine}
0<\mathfrak{g}^\omega\leq \underset{k_1\neq k_2}{\inf}\left\{\underset{\substack{x_{k_1}\in H_{k_1} \\ x_{k_2}\in H_{k_2}}}{\inf}|x_{k_1}-x_{k_2}|\right\}.
\end{equation}
It should be noted that $\|\nabla u_{1,0}\|_{L^\infty}$ grows uncontrollably as $\mathfrak{g}^\omega\to 0$.  For more details regarding this and explicit results, see~\cite{baoliyin}.

%
%
%
%
%
\section{Homogenization with Soft Inclusions}\label{section3}
%
%
%
%
%

In this section, we quantitatively discuss the convergence of solutions to~\eqref{three} as $\e,\d\to 0$ by proving Theorem~\ref{five}.  In Subsection~\ref{subsection1}, we discuss the ellipticity of $\widehat{A}_\d$ which is shown to be uniform in $\d$.  In Subsection~\ref{subsection2}, we provide the proof of Theorem~\ref{five}.

\subsection{Ellipticity of $\widehat{A}_\d$}\label{subsection1}

If $A$ satisfies~\eqref{one} and~\eqref{two}, then $\widehat{A}_\d$ defined by~\eqref{seventeen} satisfies conditions~\eqref{one} and~\eqref{two} but with possibly different constants $\widetilde{\kappa}_1$ and $\widetilde{\kappa}_2$ depending on $\kappa_1$ and $\kappa_2$ but not $\d$.  In particular, we have the following lemma. 

\begin{lemm}\label{twentyone}
Let $\widehat{A}_\d$ be defined by~\eqref{seventeen} for $0\leq\d\leq 1$.  Then
\begin{align*}
& \widehat{a}_{ij,\d}^{\a\b}(y)=\widehat{a}_{ji,\d}^{\b\a}(y)=\widehat{a}_{\a j,\d}^{i\b}(y) \\
& \widetilde{\kappa}_1|\x|^2\leq a_{ij,\d}^{\a\b}(y)\x_i^\a\x_j^\b\leq \widetilde{\kappa}_2|\x|^2
\end{align*}
for any symmetric matrix $\x=\{\x_i^\a\}$, where $\widetilde{\kappa}_1,\widetilde{\kappa}_2>0$ depend on $\kappa_1$, $\kappa_2$, and $|Q\cap\omega|$.
\end{lemm}

Lemma~\ref{twentyone} follows from the following two lemmas.  The first discusses the convergence of $\chi_{j,\d}^{\b}$ in the connected substrate for each $1\leq j,\b\leq d$ as $\d\to 0$, and the second discusses the convergence of $\widehat{A}_\d$ to $\widehat{A}_0$ as $\d\to 0$.  As $\widehat{A}_0$ is known to be uniformly elliptic (see Lemma~\ref{twentysix}), we obtain Lemma~\ref{twentyone}.

\begin{lemm}\label{twentytwo}
If $\X_0=\{\X_{j,0}^\b\}_{1\leq j,\b\leq d}$, $\X_{\d}=\{\X_{j,{\d}}^\b\}_{1\leq j,\b\leq d}$ are defined by~\eqref{eighteen} and~\eqref{nineteen}, respectively, then for $\d>0$ we have the following estimates:
\begin{itemize}
\item[(i)] $\|\one_+\nabla(\X_0-\X_\d)\|_{L^2(Q)}\leq C_1\d^{1/2}$, 
\item[(ii)] $\|\one_-\nabla\X_\d\|_{L^2(Q)}\leq C_2\d^{-1/4}$,
\end{itemize}
where $C_1$, $C_2$ depend on $\kappa_1$ and $\kappa_2$.
\end{lemm}

\begin{proof}
Let $\widetilde{\chi}_{j,0}^\b=P\chi_{j,0}^\b\in H_{}^1(Q;\R^d)$ be a periodic extension of $\chi_{j,0}^\b$ for each $1\leq j,\b\leq d$, where $P$ is the bounded linear extension operator given in~\cite[Lemma 4.1]{book2}.  Let
\[
\widetilde{\X}_{j,0}^\b(y)=y_je^\b+\widetilde{\chi}_{j,0}^\b(y).
\]
Recall that $\one_+\widetilde{\X}_0$ satisfies~\eqref{eighteen} and $\X_\d$ satisfies~\eqref{nineteen}, and so for any $\phi\in H^1_{\text{per}}(Q;\R^d)$ we have
\[
\dint_Q\lrnoe{}A\nabla (\widetilde{\X}_0-\X_\d)\cdot\nabla\phi=\d\dint_Q\one_-A\nabla\widetilde{\X}_0\cdot\nabla\phi
\]
Note
\[
\widetilde{\X}_{0}-\X_\d=\widetilde{\chi}_0-\chi_\d\in H_{\text{per}}^1(Q;\R^d),
\]
and so by the ellipticity of $A$ and Cauchy-Schwarz,
\begin{align*}
&\dint_{Q}\lrnoe{}|\nabla(\widetilde{\X}_0-\X_{\d})|^2\leq C\dint_{Q}\lrnoe{}A\nabla(\widetilde{\X}_0-\X_{\d})\cdot\nabla (\widetilde{\X}_0-\X_{\d})\\
&\hspace{20mm}=C\d\dint_{Q}\one_-\nabla\widetilde{\X}_0\cdot\nabla(\widetilde{\X}_0-\X_{\d}) \\
&\hspace{20mm}=C_1\d\dint_{Q}\one_+|\nabla\X_0|^2+\d\dint_{Q}\one_-|\nabla(\widetilde{\X}_0-\X_{\d})|^2,
\end{align*}
where $C_1$ only depends on $\kappa_1$ and $\kappa_2$.  This gives (i).  For (ii), note
\begin{align*}
\d\dint_{Q}\one_-A\nabla\X_{\d}\cdot\nabla\X_{\d}&=-\dint_{Q}\one_+A\nabla(\X_0-\X_{\d})\cdot\nabla\X_{\d} \\
&\leq C\d^{1/2}\|\one_+\nabla\X_0\|_{L^2(Q)}\|\one_+\nabla\X_{\d}\|_{L^2(Q)},
\end{align*}
where $C$ only depends on $\kappa_2$.
By (i),
\[
\d\dint_{Q}\one_-|\nabla\X_{\d}|^2\leq C\delta^{1/2}\|\one_+\nabla\X_0\|^2_{L^2(Q)},
\]
where $C$ depends on $\kappa_1$, $\kappa_2$, which gives (ii).
\end{proof}

\begin{lemm}
If $\widehat{A}_0$ and $\widehat{A}_\d$ are defined by~\eqref{seventeen} and~\eqref{twenty}, then
\[
\left| |Q\cap\omega|\widehat{A}_0-\widehat{A}_\d\right|\leq C\d^{1/2}\|\one_+\nabla\X_0\|_{L^2(Q)},
\]
where $C$ depends on $\kappa_1$ and $\kappa_2$.
\end{lemm}

\begin{proof}
Note
\[
|Q\cap\omega|\widehat{A}-\widehat{A}_\d=\dint_{Q}\one_+A\nabla(\X_0-\X_\d)-\d\dint_{Q}\one_-\nabla\X_\d,
\]
from which the desired estimate follows by Lemma~\ref{twentytwo}.
\end{proof}

\subsection{Convergence Rates}\label{subsection2}

Let $K_\e$ be defined as in Section~\ref{section2}.  Let $\eta_\e\in C_0^\infty(\Omega)$ satisfy
\begin{equation}\label{873}
\begin{cases}
0\leq \eta_\e(x)\leq 1\,\,\,\text{ for }x\in\Omega, \\
\text{supp}(\eta_\e)\subset \{x\in\Omega\,:\,\text{dist}(x,\dd\Omega)\geq 3\e\}, \\
\eta_\e=1\,\,\,\text{ on }\{x\in\Omega\,:\,\text{dist}(x,\dd\Omega)\geq 4\e\}, \\
|\nabla\eta_\e|\leq C\e^{-1}.
\end{cases}
\end{equation}
Let $\Gamma_\e=\dd\Omega\cap\e\omega$, and let $H^1(\Omega,\Gamma_\e;\R^d)$ denote the closure in $H^1(\Omega;\R^d)$ of $C^{\infty}(\R^d;\R^d)$ functions vanishing on $\Gamma_\e$.

\begin{lemm}\label{875}
Let $r_{\e,\d}=u_{\e,\d}-u_{0,\d}-\e\chi_{\d}^\e\smoothtwo{(\nabla u_{0,\d})\eta_\e}{\e}$.  Then
\begin{align*}
&\dint_{\Omega} \lr{} A^\e\nabla r_{\e,\d}\cdot\nabla w \\
&\hspace{10mm}=\dint_{\Omega}(\eta_\e-1)\lr{}A^\e\nabla \left[u_{\e,\d}-u_{0,\d}\right]\cdot\nabla w+\dint_{\Omega}\lr{}A^\e\nabla \left[u_{\e,\d}-u_{0,\d}\right]\cdot\left[w\nabla \eta_\e\right] \\
&\hspace{20mm}+\dint_{\Omega}\left[\widehat{A}_\d-k_{\d}^\e A^\e\right]\left[\nabla u_{0,\d}-\smoothtwo{(\nabla u_{0,\d})\eta_\e}{\e}\right]\cdot\nabla w \\
&\hspace{20mm}-\dint_{\Omega}\left[\widehat{A}_\d-k_{\d}^\e A^\e\nabla \X_\d\right]\smoothtwo{(\nabla u_{0,\d})\eta_\e}{\e}\cdot\nabla w \\
&\hspace{20mm}-\e\dint_{\Omega}\lr{2}A^\e\chi_\d^\e\nabla\smoothtwo{(\nabla u_{0,\d})\eta_\e}{\e}\cdot\nabla w\\
\end{align*}
for any $w\in H^1(\Omega,\Gamma_\e;\R^d)$.
\end{lemm}

\begin{proof}
Since $u_{\e,\d}$ and $u_{0,\d}$ solve~\eqref{three} and~\eqref{874}, respectively,
\[
\dint_{\Omega}\lr{} A^\e\nabla u_{\e,\d}\cdot\nabla [w\eta_\e]=\dint_{\Omega}\widehat{A}_\d\nabla u_{0,\d}\cdot\nabla [w\eta_\e]=0
\]
for any $w\in H^1(\Omega,\Gamma_\e;\R^d)$, where $\eta_\e$ denotes the cuttoff function defined by~\eqref{873}.  Hence,
\begin{align*}
&\dint_{\Omega}\lr{}A^\e\nabla r_{\e,\d}\cdot\nabla w \\
&\hspace{10mm}= \dint_{\Omega}\lr{}A^\e\nabla u_{\e,\d}\cdot\nabla w-\dint_{\Omega}\lr{}A^\e\nabla u_{0,\d}\cdot\nabla w \\
&\hspace{20mm}-\dint_{\Omega}\lr{}A^\e\nabla\left[\e\chi_\d^\e\smoothtwo{(\nabla u_{0,\d})\eta_\e}{\e}\right]\cdot\nabla w \\
&\hspace{10mm}=\dint_{\Omega}(\eta_\e-1)\lr{}A^\e\nabla u_{\e,\d}\cdot\nabla w+\dint_{\Omega}\lr{}A^\e\nabla u_{\e,\d}\cdot\left[w\nabla \eta_\e\right] \\
&\hspace{20mm} +\dint_{\Omega}\lr{}A^\e\nabla u_{0,\d}\cdot\nabla w -\dint_{\Omega}\lr{}A^\e\nabla \chi_\d^\e\smoothtwo{(\nabla u_{0,\d})\eta_\e}{\e}\cdot\nabla w\\
&\hspace{20mm}-\e\dint_{\Omega}\lr{}A^\e\chi_\d^\e\nabla\smoothtwo{(\nabla u_{0,\d})\eta_\e}{\e}\cdot\nabla w\\
&\hspace{10mm}=\dint_{\Omega}(\eta_\e-1)\lr{}A^\e\nabla \left[u_{\e,\d}-u_{0,\d}\right]\cdot\nabla w+\dint_{\Omega}\lr{}A^\e\nabla \left[u_{\e,\d}-u_{0,\d}\right]\cdot\left[w\nabla \eta_\e\right] \\
&\hspace{20mm} +\dint_{\Omega}\left[\widehat{A}_\d-\lr{}A^\e\right]\nabla u_{0,\d}\cdot\nabla w -\dint_{\Omega}\lr{}A^\e\nabla \chi_\d^\e\smoothtwo{(\nabla u_{0,\d})\eta_\e}{\e}\cdot\nabla w\\
&\hspace{20mm}-\e\dint_{\Omega}\lr{}A^\e\chi_\d^\e\nabla\smoothtwo{(\nabla u_{0,\d})\eta_\e}{\e}\cdot\nabla w\\
&\hspace{10mm}=\dint_{\Omega}(\eta_\e-1)\lr{}A^\e\nabla \left[u_{\e,\d}-u_{0,\d}\right]\cdot\nabla w+\dint_{\Omega}\lr{}A^\e\nabla \left[u_{\e,\d}-u_{0,\d}\right]\cdot\left[w\nabla \eta_\e\right] \\
&\hspace{20mm}+\dint_{\Omega}\left[\widehat{A}_\d-\lr{} A^\e\right]\left[\nabla u_{0,\d}-\smoothtwo{(\nabla u_{0,\d})\eta_\e}{\e}\right]\cdot\nabla w \\
&\hspace{20mm}-\dint_{\Omega}\left[\widehat{A}_\d-\lr{} A^\e-\lr{}A^\e\nabla \chi_\d^\e\right]\smoothtwo{(\nabla u_{0,\d})\eta_\e}{\e}\cdot\nabla w \\
&\hspace{20mm}-\e\dint_{\Omega}\lr{}A^\e\chi_\d^\e\nabla\smoothtwo{(\nabla u_{0,\d})\eta_\e}{\e}\cdot\nabla w,
\end{align*}
where we have used the equalities
\[
\nabla w=(1-\eta_\e)\nabla w-\nabla [w\eta_\e]+w\nabla \eta_\e
\]
and
\[
\dint_{\Omega}\widehat{A}_\d\nabla u_{0,\d}\cdot\nabla w=\dint_{\Omega}(1-\eta_\e)\widehat{A}_\d\nabla u_{0,\d}\cdot\nabla w+\dint_{\Omega}\widehat{A}_\d\nabla u_{0,\d}\cdot[w\nabla \eta_\e].
\]
This proves the lemma.
\end{proof}

\begin{lemm}\label{}
For $w\in H^1(\Omega,\Gamma_\e;\R^d)$,
\begin{align*}
\left|\dint_{\Omega}\lr{}A^\e\nabla r_{\e,\d}\cdot\nabla w\right| &\leq C\left\{\|\nabla u_{0,\d}\|_{L^2(\mathcal{O}_{4\e})}+\|(\nabla u_{0,\d})\eta_\e-\smoothone{(\nabla u_{0,\d})\eta_\e}{\e}\|_{L^2(\Omega)}\right. \\
&\hspace{10mm}\left.+\e\|\smoothone{(\nabla^2 u_{0,\d})\eta_\e}{\e}\|_{L^2(\Omega)}+\|\lr{}\nabla u_{\e,\d}\|_{L^2(\mathcal{O}_{4\e})}\right\}\|\nabla w\|_{L^2(\Omega)}
\end{align*}
\end{lemm}

\begin{proof}
By Lemma~\ref{875},
\begin{equation}\label{}
\dint_{\Omega}\lr{}A^\e\nabla r_{\e,\d}\cdot\nabla w=I_1+I_2+I_3+I_4+I_5,
\end{equation}
where
\begin{align*}
I_1 &= \dint_{\Omega}(\eta_\e-1)\lr{}A^\e\nabla \left[u_{\e,\d}-u_{0,\d}\right]\cdot\nabla w\\
I_2 &= \dint_{\Omega}\lr{}A^\e\nabla \left[u_{\e,\d}-u_{0,\d}\right]\cdot[w\nabla \eta_\e]\\
I_3 &= \dint_{\Omega}\left[\widehat{A}_\d-\lr{} A^\e\right]\left[\nabla u_{0,\d}-\smoothtwo{(\nabla u_{0,\d})\eta_\e}{\e}\right]\cdot\nabla w\\
I_4 &= -\dint_{\Omega}\left[\widehat{A}_\d-\lr{} A^\e\nabla \X^\e_\d\right]\smoothtwo{(\nabla u_{0,\d})\eta_\e}{\e}\cdot\nabla w \\
I_5 &= -\e\dint_{\Omega}\lr{}A^\e\chi_\d^\e\nabla\smoothtwo{(\nabla u_{0,\d})\eta_\e}{\e}\cdot\nabla w\\
\end{align*}
and $w\in H^1(\Omega,\Gamma_\e;\R^d)$.  
Since $\text{supp}(1-\eta_\e)\subset\mathcal{O}_{4\e}$, where
\[
\mathcal{O}_{4\e}=\{x\in\Omega\,:\,\text{dist}(x,\dd\Omega)<4\e\},
\]
by Cauchy-Schwarz,~\eqref{873}, and~\eqref{two} we have
\begin{equation}\label{876}
|I_1|\leq C\left\{\|\nabla u_0\|_{L^2(\mathcal{O}_{4\e})}+\|\lr{}\nabla u_{\e,\d}\|_{L^2(\mathcal{O}_{4\e})}\right\}\|\nabla w\|_{L^2(\Omega)}.
\end{equation}
Similarly, as $\text{supp}(\nabla\eta_\e)\subset\mathcal{O}_{4\e}$, Cauchy-Schwarz,~\cite[Lemma 3.4]{bcr17}, and~\eqref{873} imply
\begin{equation}\label{877}
|I_2|\leq C\left\{\|\nabla u_0\|_{L^2(\mathcal{O}_{4\e})}+\|\lr{}\nabla u_{\e,\d}\|_{L^2(\mathcal{O}_{4\e})}\right\}\|\nabla w\|_{L^2(\Omega)}.
\end{equation}
Using~\eqref{873} again,
\begin{align*}
&\|\nabla u_0-\smoothtwo{(\nabla u_0)\eta_\e}{\e}\|_{L^2(\Omega)} \\ 
&\hspace{10mm}\leq \|(1-\eta_\e)\nabla u_0\|_{L^2(\Omega)}+\|(\nabla u_0)\eta_\e-\smoothone{(\nabla u_0)\eta_\e}{\e}\|_{L^2(\Omega)} \\
&\hspace{20mm}+\|\smoothone{(\nabla u_0)\eta_\e-\smoothone{(\nabla u_0)\eta_\e}{\e}}{\e}\|_{L^2(\Omega)} \\
&\hspace{10mm}\leq\|\nabla u_0\|_{L^2(\mathcal{O}_{4\e})}+C\|(\nabla u_0)\eta_\e-\smoothone{(\nabla u_0)\eta_\e}{\e}\|_{L^2(\Omega)}.
\end{align*}
Therefore, by Cauchy-Schwarz,
\begin{align}
|I_3|&\leq C\|\nabla u_0-\smoothtwo{(\nabla u_0)\eta_\e}{\e}\|_{L^2(\Omega)}\|\nabla w\|_{L^2(\Omega)} \nonumber\\
&\leq C\left\{\|\nabla u_0\|_{L^2(\mathcal{O}_{4\e})} \right. \nonumber\\
&\hspace{10mm}\left.+\|(\nabla u_0)\eta_\e-\smoothone{(\nabla u_0)\eta_\e}{\e}\|_{L^2(\Omega)} \right\}\|\nabla w\|_{L^2(\Omega)}.\label{878}
\end{align}

Set $B_\d=\widehat{A}_\d-\lrnoe{}A\nabla\X_\d$.  By~\eqref{nineteen} and~\eqref{twenty}, $B_\d=\{b_{ij,\d}^{\a\b}\}$ satisfies the assumptions of Lemma~\ref{fourteen}.  Therefore, there exists $\pi_\d=\{\pi_{kij,\d}^{\a\b}\}$ that is 1-periodic with
\[
\dfrac{\dd}{\dd y_k}\pi_{kij,\d}^{\a\b}=b_{ij,\d}^{\a\b}\,\,\,\text{ and }\,\,\,\pi_{kij,\d}^{\a\b}=-\pi_{ikj,\d}^{\a\b},
\]
where
\[
b_{ij,\d}^{\a\b}=\widehat{a}_{ij,\d}^{\a\b}-\lrnoe{2}a_{ik,\d}^{\a\g}\dfrac{\dd}{\dd y_k}\X_{j,\d}^{\g\b}.
\]
Moreover, $\|\pi_{ij,\d}^{\a\b}\|_{H^1(Q)}\leq C$ for some constant $C$ depending on $\kappa_1$, $\kappa_2$, but not $\d$ given Lemma~\ref{twentyone}.  Hence, integrating by parts gives
\begin{align*}
\dint_{\Omega}b_{ij,\d}^{\a\b\e}\smoothtwo{\dfrac{\dd u_{0,\d}^\b}{\dd x_j}\eta_\e}{\e}\dfrac{\dd\widetilde{w}^\a}{\dd x_i} &=-\e\dint_{\Omega}\pi_{kij,\d}^{\a\b\e}\dfrac{\dd}{\dd x_k}\left[\smoothtwo{\dfrac{\dd u_{0,\d}^\b}{\dd x_j}\eta_\e}{\e}\dfrac{\dd{w}^\a}{\dd x_i}\right] \\
&=-\e\dint_{\Omega}\pi_{kij,\d}^{\a\b\e}\dfrac{\dd}{\dd x_k}\left[\smoothtwo{\dfrac{\dd u_{0,\d}^\b}{\dd x_j}\eta_\e}{\e}\right]\dfrac{\dd{w}^\a}{\dd x_i} ,
\end{align*}
since
\[
\dint_{\Omega}\pi_{kij,\d}^{\a\b\e}\smoothtwo{\dfrac{\dd u_{0,\d}^\b}{\dd x_j}\eta_\e}{\e}\dfrac{\dd^2{w}^\a}{\dd x_k\dd x_i}=0
\]
due to the anit-symmetry of $\pi_\d$.  Thus, by Lemma~\ref{twentyseven}, and~\eqref{873},
\begin{align}
|I_4|&\leq C\e\|\pi_\d^\e\nabla\smoothtwo{(\nabla u_{0,\d})\eta_\e}{\e}\|_{L^2(\Omega)}\|\nabla w\|_{L^2(\Omega)} \nonumber\\
&\leq C\left\{\|\nabla u_{0,\d}\|_{L^2(\mathcal{O}_{4\e})}+\e\|\smoothone{(\nabla^2 u_{0,\d})\eta_\e}{\e}\|_{L^2(\Omega)}\right\}\|\nabla w\|_{L^2(\Omega)}.\label{879}
\end{align}

Finally, by Lemma~\ref{twentyseven} and~\eqref{873},
\begin{align}
|I_5|\leq C\left\{\|\nabla u_{0,\d}\|_{L^2(\mathcal{O}_{4\e})}+\e\|\smoothone{(\nabla^2 u_{0,\d})\eta_\e}{\e}\|_{L^2(\Omega)}\right\}\|\nabla w\|_{L^2(\Omega)}\label{880}
\end{align}
The desired estimate follows from~\eqref{876}--\eqref{880}.
\end{proof}

\begin{lemm}\label{892}
For $w\in H^1(\Omega,\Gamma_\e;\R^d)$,
\[
\left|\dint_{\Omega}\lr{}A^\e\nabla r_{\e,\d}\cdot\nabla w\right|\leq C\e^{\mu}\|f\|_{H^1(\dd\Omega)}\|\nabla w\|_{L^2(\Omega)},
\]
where $\mu>0$ depends on $d$, $\kappa_1$, and $\kappa_2$.
\end{lemm}

\begin{proof}
Recall that $u_{0,\d}$ satisfies $\mathcal{L}_{0,\d}(u_{0,\d})=0$ in $\Omega$, and so it follows from estimates for solutions in Lipschitz domains to constant-coefficient systems that
\begin{equation}\label{881}
\|(\nabla u_{0,\d})^*\|_{L^2(\dd\Omega)}\leq C\|f\|_{H^1(\dd\Omega)},
\end{equation}
where $(\nabla u_{0,\d})^*$ denotes the nontangential maximal function of $\nabla u_{0,\d}$ (see~\cite{dahlberg}).  By the coarea formula,
\begin{equation}\label{882}
\|\nabla u_{0,\d}\|_{L^2(\mathcal{O}_{4\e})}\leq C\e^{1/2}\|(\nabla u_{0,\d})^*\|_{L^2(\dd\Omega)}\leq C\e^{1/2}\|f\|_{H^1(\dd\Omega)}.
\end{equation}

Notice that if $u_{0,\d}$ solves~\eqref{874}, then $\mathcal{L}_{0,\d}(\nabla u_{0,\d})=0$ in $\Omega$, and so we may use an interior Lipschitz estimate for $\mathcal{L}_{0,\d}$.  That is,
\begin{equation}\label{883}
|\nabla^2 u_{0,\d}(x)|\leq\dfrac{C}{\rho(x)}\left(\dashint_{B(x,\rho(x)/8)}|\nabla u_{0,\d}|^2\right)^{1/2},
\end{equation}
where $\rho(x)=\text{dist}(x,\dd\Omega)$.  In particular,
\begin{align}
\|(\nabla^2 u_{0,\d})\eta_\e\|_{L^2(\Omega)} &\leq \left(\dint_{\Omega\backslash\mathcal{O}_{3\e}}|\nabla^2 u_{0,\d}|^2\right)^{1/2} \nonumber\\
&\leq C\left(\dint_{\Omega\backslash\mathcal{O}_{3\e}}\dashint_{B(x,\rho(x)/8)}\left|\dfrac{\nabla u_{0,\d}(y)}{\rho(x)}\right|^2\,dy\>dx\right)^{1/2} \nonumber\\
&\leq C\left(\dint_{3\e}^{C_0}t^{-2}\dint_{\dd\mathcal{O}_t\cap\Omega}\dashint_{B(x,t/8)}|\nabla u_{0,\d}(y)|^2\,dy \>dS(x)\>dt\right)^{1/2} \nonumber\\
&\hspace{30mm}+C_1\left(\dint_{\Omega\backslash\mathcal{O}_{C_0}}|\nabla u_{0,\d}|^2\right)^{1/2}\nonumber\\
&\leq C\|(\nabla u_{0,\d})^*\|_{L^2(\dd\Omega)}\left(\dint_{3\e}^{C_0}t^{-2}\,dt\right)^{1/2}+C_1\|\nabla u_{0,\d}\|_{L^2(\Omega)} \nonumber\\
&\leq C\left\{\e^{-1/2}\|f\|_{H^1(\dd\Omega)}+\|f\|_{H^{1/2}(\dd\Omega)}\right\} \nonumber\\
&\leq C_0\e^{-1/2}\|f\|_{H^1(\dd\Omega)}\label{}.
\end{align}
where $C_0$ is a constant depending on $\Omega$, and we have used~\eqref{873},~\eqref{881}~\eqref{882}, the coarea formula, energy estimates, and~\eqref{883}.  Hence,
\begin{equation}\label{884}
\e\|\smoothone{(\nabla^2 u_{0,\d})\eta_\e}{\e}\|_{L^2(\Omega)}\leq C\e^{1/2}\|f\|_{H^1(\dd\Omega)}.
\end{equation}

By Lemma~\ref{thirtysix},
\begin{equation}\label{887}
\|(\nabla u_{0,\d})\eta_\e-\smoothone{(\nabla u_{0,\d})\eta_\e}{\e}\|_{L^2(\Omega)}\leq C\e^{1/2}\|f\|_{H^1(\dd\Omega)}.
\end{equation}
where the last inequality follows from~\eqref{884} and~\eqref{873}.

Finally, we establish a $W^{1,p}$-estimate for some $p>2$ for $u_{\e,\d}$ uniform in $\e$ and $\d$ by establishing a reverse H\"older inequality.  Indeed, if there exists a $p>2$ so that
\[
\left(\int_{\Omega}|\lr{}\nabla u_{\e,\d}|^p\right)^{1/p}\leq C\|f\|_{H^1(\dd\Omega)},
\]
then H\"older's inequality implies
\begin{equation}\label{886}
\dint_{\mathcal{O}_{4\e}}|\lr{}\nabla u_{\e,\d}|^2\leq C\e^{(p-2)/p}\|f\|_{H^1(\Omega)}^2.
\end{equation}
The existence of such a $p$ follows from the Lemma~\ref{885}.  Equations~\eqref{882}, \eqref{884},~\eqref{887}, and~\eqref{886} give the desired result.
\end{proof}

\begin{lemm}\label{885}
There exists a $p_0>2$ such that
\[
\left(\dint_{\Omega}|\lr{}\nabla u_{\e,\d}|^{p_0}\right)^{1/p_0}\leq C\|f\|_{H^1(\dd\Omega)}
\]
for some constant $C$ depending on $\kappa_1$, $\kappa_2$, $d$, $p_0$, and $\Omega$.
\end{lemm}

\begin{proof}
The desired estimate essentially follows from Cacciopoli's inequality, the Poincar\'e-Sobolev inequality, and the self-improving property of reverse H\"older inequalities.  We prove an interior estimate, and the boundary estimate follows with an analogous proof.

Take $B(x_0,2r)\subset\Omega$, and note that Cacciopoli's inequality (see Lemma~\ref{fortysix}) implies
\begin{align*}
&\left(\dashint_{B(x_0,r)}|\lr{}\nabla u_{\e,\d}|^2\right)^{1/2}\\
&\hspace{10mm}\leq \dfrac{C}{r}\left(\dashint_{B(x_0,2r)}|\lr{} u_{\e,\d}|^2\right)^{1/2} \\
&\hspace{10mm}\leq\dfrac{C}{r}\left\{\d\left(\dashint_{B(x_0,2r)}|u_{\e,\d}|^2\right)^{1/2}+\left(\dashint_{B(x_0,2r)}|P_\e(\one_+^\e u_{\e,\d})|^2\right)^{1/2}\right\},
\end{align*}
which is invariant if we subtract a constant vector from $u_{\e,\d}$.  If we subtract the average value of $u_{\e,\d}$ over the ball $B(x_0,2r)$, then by the Poincar\'e-Sobolev ineqaulity
\begin{align*}
&\left(\dashint_{B(x_0,r)}|\lr{}\nabla u_{\e,\d}|^2\right)^{1/2} \\
&\hspace{10mm}\leq\d\left(\dashint_{B(x_0,2r)}|\nabla u_{\e,\d}|^{s}\right)^{1/s}+\dfrac{C}{r}\left(\dashint_{B(x_0,2r)}|P_\e(\one_+^\e u_{\e,\d})|^2\right)^{1/2},
\end{align*}
where $s=2d/(d+2)$.  Similarly, by subtracting another constant we can show
\begin{align*}
&\left(\dashint_{B(x_0,r)}|\lr{}\nabla u_{\e,\d}|^2\right)^{1/2} \\
&\hspace{10mm}\leq\d\left(\dashint_{B(x_0,2r)}|\nabla u_{\e,\d}|^{s}\right)^{1/s}+\left(\dashint_{B(x_0,2r)}|\nabla P_\e(\one_+^\e u_{\e,\d})|^s\right)^{1/s},
\end{align*}
which by Lemma~\ref{twentyfive} shows
\[
\left(\dashint_{B(x_0,r)}w^{q}\right)^{1/q}\leq C\dashint_{B(x_0,2r)}w,
\]
where $w=|\lr{}\nabla u_{\e,\d}|^{s}$ and $q=2/s$.  By the self-improving property of reverse H\"older inequalities (see~\cite[Chapter V, Proposition 1.1]{giaquinta}),
\[
\left(\dashint_{B(x_0,r)}w^{t}\right)^{1/t}\leq C\left(\dashint_{B(x_0,2r)}w^q\right)^{1/q},
\]
for any $t\in [q,q+\nu)$ for some $\nu>0$ depending on $\kappa_1$, $\kappa_2$, and $d$.  That is,
\begin{equation}\label{889}
\left(\dashint_{B(x_0,r)}|\lr{}\nabla u_{\e,\d}|^p\right)^{1/p}\leq C\left(\dashint_{B(x_0,2r)}|\lr{}\nabla u_{\e,\d}|^2\right)^{1/2}
\end{equation}
for any $p\in [2,2+\nu)$ and any $B(x_0,2r)\subset\Omega$.

We may show a similar estimate for any ball $B(x_0,2r)$ with $x_0\in\dd\Omega$.  That is, if $F=f$ on $\dd\Omega$ and $F\in H^{3/2}(\Omega)$, then the continuous injection $H^{3/2}(\Omega)\hookrightarrow W^{1,q}(\Omega)$ for any $q\geq 2d/(d-1)$ gives the estimate
\begin{align}
&\left(\dashint_{B(x_0,r)\cap\Omega}|\lr{}\nabla u_{\e,\d}|^p\right)^{1/p} \nonumber\\
&\hspace{10mm}\leq C\left(\dashint_{B(x_0,2r)\cap\Omega}|\lr{}\nabla u_{\e,\d}|^2\right)^{1/2}+\left(\dashint_{B(x_0,2r)\cap\Omega}|\nabla F|^q\right)^{1/q}.\label{890}
\end{align}
Patching together inequalities~\eqref{889} and~\eqref{890} gives the desired estimate for some $p_0>2$.
\end{proof}

\begin{proof}[Proof of Theorem~\ref{twelve}]

Note $\d r_{\e,\d}\in H_0^1(\Omega;\R^d)\subset H^1(\Omega,\Gamma_\e;\R^d)$, and so by Lemmas~\ref{892} and~\eqref{two},
\begin{align*}
\|\d e(r_{\e,\d})\|^2_{L^2(\Omega)} &\leq C\d\dint_{\Omega}\lr{}A^\e\nabla r_{\e,\d}\cdot\nabla r_{\e,\d} \\
&\leq C\e^{\mu}\|f\|_{H^1(\dd\Omega)}\|\d\nabla r_{\e,\d}\|_{L^2(\Omega)},
\end{align*}
where $e(r_{\e,\d})$ denotes the symmetric part of $\nabla r_{\e,\d}$.  Korn's first inequality then implies
\[
\|\d\nabla r_{\e,\d}\|^2_{L^2(\Omega)}\leq C\|\d e(r_{\e,\d})\|^2_{L^2(\Omega)}\leq C\e^{\mu}\|f\|_{H^1(\dd\Omega)}\|\d\nabla r_{\e,\d}\|_{L^2(\Omega)},
\]
and so
\begin{equation}\label{901}
\|\d\nabla r_{\e,\d}\|_{L^2(\Omega)}\leq C\e^{\mu}\|f\|_{H^1(\dd\Omega)}.
\end{equation}

Note also $P_\e(\one_+^\e r_{\e,\d})\in H^1(\Omega,\Gamma_\e;\R^d)$, and so by Lemmas~\ref{892} and~\ref{twentyfive},
\begin{align*}
&\|\one_+^\e e[P_\e(\one_+^\e r_{\e,\d})]\|^2_{L^2(\Omega)} \\
&\hspace{10mm}\leq C\dint_{\Omega}\lr{}A^\e\nabla r_{\e,\d}\cdot\nabla P_\e(\one_+^\e r_{\e,\d})-\d\dint_{\Omega}\one_-^\e A^\e\nabla r_{\e,\d}\cdot\nabla P_\e(\one_+^\e r_{\e,\d}) \\
&\hspace{10mm}\leq C\e^{\mu}\|f\|_{H^1(\dd\Omega)}\|\one_+^\e\nabla r_{\e,\d}\|_{L^2(\Omega)},
\end{align*}
where we've used~\eqref{901}.  Korn's first inequality for periodically perforated domains then implies
\[
\|\one_+^\e\nabla r_{\e,\d}\|^2_{L^2(\Omega)}\leq C\|\one_+^\e e[P_\e(\one_+^\e r_{\e,\d})]\|^2_{L^2(\Omega)}\leq C\e^{\mu}\|f\|_{H^1(\dd\Omega)}\|\one_+^\e\nabla r_{\e,\d}\|_{L^2(\Omega)},
\]
and so
\begin{equation}\label{902}
\|\one_+^\e\nabla r_{\e,\d}\|_{L^2(\Omega)}\leq C\e^{\mu}\|f\|_{H^1(\dd\Omega)}.
\end{equation}
Equations~\eqref{901} and~\eqref{902} give the desired estimate.  Indeed,
\[
\|\lr{}\nabla r_{\e,\d}\|_{L^2(\Omega)}\leq \|\one_+^\e \nabla r_{\e,\d}\|_{L^2(\Omega)}+\|\d\nabla r_{\e,\d}\|_{L^2(\Omega)} \leq C\e^{\mu}\|f\|_{H^1(\dd\Omega)}
\]
\end{proof}

\section{Interior Estimates at the large scale}\label{section4}
%
%
%
%
%

In this section, we discuss \textit{a priori} interior estimates for the boundary value problem~\eqref{three} at the macroscopic scale by proving Theorem~\ref{five}.  By macroscopic, we refer to the case when $\e/r\leq 1$.  Throughout this section, let $B(r)\equiv B(x_0,r)$ denote the ball of radius $r>0$ centered at some $x_0\in\R^d$.

The following lemma is essentially Cacciopoli's inequality for the operator $\mathcal{L}_{\e,\d}$ defined by~\eqref{fortythree}.  The proof is similar to a proof of the classical Cacciopoli's inequality, but nevertheless we present a proof for completeness.

\begin{lemm}\label{fortysix}
Suppose $\mathcal{L}_{\e,\d}(u_{\e,\d})=0$ in $B(2r)$ for some $r>0$.  Then
\[
\left(\dashint_{B(r)}|\lr{}{\nabla u_{\e,\d}}|^2\right)^{1/2}\leq \dfrac{C}{r}\left(\dashint_{B(2r)}|\lr{}{u_{\e,\d}}|^2\right)^{1/2}
\]
where $C$ depends only on $\kappa_1$ and $\kappa_2$.
\end{lemm}

\begin{proof}
By rescaling we may assume $r=1$, i.e., set $U(x)=u_{\e,\d}(rx)$ and note $U$ satisfies $\mathcal{L}_{\e/r,\d}(U)=0$ in $B(2)$.  Let $\z\in C_0^\infty (B(2))$.  Then
\begin{align}
0 &= \dint_{B(2)}\lr{}A_{}^\e\nabla u\cdot\nabla (u\z^2) \nonumber\\
&\geq \kappa_1\dint_{B(2)}\lr{}|\nabla u|^2\z^2-2\dint_{B(2)}(u\z)\lr{}A_{}^\e\nabla u\cdot\nabla \z, \label{fortyseven}
\end{align}
where $u\equiv u_{\e,\d}$.  Equation~\eqref{fortyseven}, $\d\leq 1$, and Cauchy-Schwarz imply
\begin{align*}
&\dint_{B(2)}\one_-^\e|\d\nabla u|^2\z^2\leq \dfrac{C_1}{\g}\dint_{B(2)}\lr{2}|\nabla u|^2\z^2+\g C_2\dint_{B(2)}\lr{2}u^2|\nabla\z|^2
\end{align*}
for any $\g>0$.  Similarly, equation~\eqref{fortyseven} and Cauchy-Schwarz give
\[
\dint_{B(2)}\one_+^\e|\nabla u|^2\zeta^2\leq \dfrac{C_1}{\g'}\dint_{B(2)}\lr{2}|\nabla u|^2\z^2+\g' C_2\dint_{B(2)}\lr{2}u^2|\nabla\z|^2
\]
for any $\g'>0$.  Choosing $\g$, $\g'$ large enough gives
\[
\dint_{B(2)}|\lr{}\nabla u|^2\z^2\leq C\dint_{B(2)}|\lr{} u|^2|\nabla \z^2|
\]
for some constant $C$ depending on $\kappa_1$ and $\kappa_2$.  Choose $\z$ so that $\z\equiv 1$ in $B(1)$ and $|\nabla \z|\leq C$.  The desired inequality follows.
\end{proof}

\begin{lemm}\label{fiftyone}
Suppose $\mathcal{L}_{\e,\d}(u_{\e,\d})=0$ in $B(3r)$.  There exists $v\in H^1(B(r);\R^d)$ satisfying $\mathcal{L}_{0,\d}(v)=0$ in $B(r)$ and 
\begin{align*}
&\left(\dashint_{B(r)}|\lr{}({u_{\e,\d}-v})|^2\right) \leq C\left(\dfrac{\e}{r}\right)^{\mu}\left(\dashint_{B(3r)}|\lr{}{u_{\e,\d}}|^2\right)^{1/2},
\end{align*}
where $C$ depends on $\kappa_1$, $\kappa_2$, and $d$ and $\mu>0$.
\end{lemm}

\begin{proof}
First we prove the lemma for $r=1$.  By Lemma~\ref{fortysix} and estimate~\eqref{fortynine} in Theorem~\ref{twentyfive} of Section~\ref{section2},
\begin{align*}
&\left(\dashint_{B(3/2)}|\nabla P_\e(\one_+^\e u_{})|^2\right)^{1/2}+\left(\dashint_{B(3/2)}|\d\nabla u_{}|^2\right)^{1/2}\leq C\left(\dashint_{B(3)}|\lr{}{u_{}}|^2\right)^{1/2},
\end{align*}
where $u\equiv u_{\e,\d}$.  Specifically, there exists a $t\in [1,5/4]$ such that
\begin{equation}\label{fiftyfive}
\|P_\e(\one_+^\e u_{})\|_{H^1(\dd B(t))}+\d\|u_{}\|_{H^1(\dd B(t))}\leq C\|\lr{}{u_{}}\|_{L^2(B(3))}.
\end{equation}
Let $v$ denote the weak solution to the Dirichlet problem $\mathcal{L}_{0,\d}(v)=0$ in $B(t)$ and $v=P_\e(\one_+^\e u_{})$ on $\dd B(t)$.  Note $v=u=P_\e(\one_+^\e u_{})$ on $\dd B(t)\cap\e\omega$.  By Theorem~\ref{twelve},
\begin{align}
\|\lr{}({u_{}-v})\|_{L^2(B(1))} &\leq C\e^{\mu}\|P_\e(\one_+^\e u_{})\|_{H^1(\dd B(t))} \nonumber\\
&\hspace{15mm}+\d\|\nabla P_\e(\one_+^\e u_{})-\nabla u_{}\|_{L^{2}(B(t))}\label{eightbajillion},
\end{align}
since
\begin{align*}
\|\lr{}\chi_\d^\e\smoothtwo{(\nabla v)\eta_\e}{\e}\|_{L^2(B(t))} &\leq C\|\nabla v\|_{L^2(B(t))} \\
&\leq C\|P_\e(\one_+^\e u)\|_{H^1(\dd B(t))},
\end{align*}
where we've used notation consistent with Theorem~\ref{twelve}.  

By Lemma~\ref{fortysix},
\begin{align}\label{seventysix}
\dashint_{B(t)}|\lr{}\nabla w|^2\leq \dashint_{B(t)}|\nabla P_\e(\one_+^\e u)|^2+C\dashint_{B(2t)}|\lr{}w|^2,
\end{align}
where $w=P_\e(\one_+^\e u)-u$.  Equation~\eqref{seventysix} follows from the fact that $\mathcal{L}_{\e,\d}(w)=\mathcal{L}_{\e,\d}(P_\e(\one_+^\e u))$ in $B(3)$ and $t\in [1,5/4]$.  Note by Lemma~\ref{twentyfive}, $w=0$ a.e. in $B(3)\cap\e\omega$.  Hence, Poincar\'e's inequality gives
\begin{align}\label{seventyeightbajillion}
\dashint_{B(2t)}|\lr{}w|=\d^2\dashint_{B(2t)}\one_-^\e|w|^2 \leq \e^2\dashint_{B(3)}|\lr{}\nabla w|^2.
\end{align}
Indeed, set $W(x)=w(\e x)$, and let $\{H_k\}_{k=1}^{N(\e)}$ denote the bounded, connected components of $\R^d\backslash\omega$ with $\e H_k\cap B(2t)\neq \emptyset$.  Then $W=0$ on $\dd H_k$ for each $k$, and so
\begin{align*}
\dint_{B(2t)}|w|^2 \leq \sum_{k=1}^N\dint_{H_k}|W|^2 \leq C\e^2\sum_{k=1}^N\dint_{\e H_k}|\nabla w|^2\leq C\e^2\dint_{B(3)}|\nabla w|^2,
\end{align*}
where $C$ is independent of $\e$ since $\omega$ is periodic.  Lemma~\ref{twentyfive} together with~\eqref{fiftyfive},~\eqref{eightbajillion} and~\eqref{seventyeightbajillion} give the estimate for $r=1$.

Now we prove the estimate for arbitrary $r>0$.  To this end, let $U(x)=u(rx)$, and note $\mathcal{L}_{\e/r,\d}(U)=0$ in $B(3)$.  By the above, there exists a $V\in H^1(B(1);\R^d)$ satisfying $\mathcal{L}_{0,\d}(V)=0$ in $B(1)$ and
\[
\left(\dashint_{B(1)}|k_{\d}^{\e/r}({U-V})|^2\right) \leq C\left(\dfrac{\e}{r}\right)^{\mu}\left(\dashint_{B(3)}|k_{\d}^{\e/r}{U}|^2\right)^{1/2},
\]
The change of variables $r x\mapsto x$ gives the desired estimate.

\end{proof}

\begin{lemm}\label{fiftytwo}
Suppose $\mathcal{L}_{0,\d}(v)=0$ in $B(2r)$.  Then for $r\geq \e$,
\begin{equation}\label{fiftythree}
\left(\dashint_{B(r)}|v|^2\right)^{1/2}\leq C\left(\dashint_{B(2r)}|\lr{}{v}|^2\right)^{1/2}
\end{equation}
for a constant $C$ depending on $\omega$, $\kappa_1$, $\kappa_2$, and $d$.
\end{lemm}

\begin{proof}
See~\cite{bcr17} for a proof when $\d=0$.  The case $\d>0$ follows similary given Lemma~\ref{twentyone}.
\end{proof}

For $w\in L_{\text{loc}}^2(B(r);\R^d)$, $\d\geq 0$, and $\e,r>0$, set
\begin{equation}\label{fortyfive}
H_{\e,\d}(r;w)=\dfrac{1}{r}\underset{\substack{M\in\R^{d\times d} \\ q\in\R^d}}{\inf}\left(\dashint_{B(r)}|\lr{}(w-Mx-q)|^2\right)^{1/2}.
\end{equation}

\begin{lemm}\label{fiftyfour}
Suppose $v$ satisfies $\mathcal{L}_{0,\d}(v)=0$ in $B(1)$.  For any $r\in [\e,1]$ and $\theta\in (0,1/4)$,
\[
H_{\e,\d}(\theta r;v)\leq C\theta H_{\e,\d}(r;v)
\]
for some constant $C$ depending on $d$, $\kappa_1$, $\kappa_2$, and $\omega$.
\end{lemm}

\begin{proof}
It follows from interior $C^2$-estimates for elasticity systems with constant coefficients that for any $\theta\in (0,1/4)$,
\[
H_{\e,\d}(\theta r; v)\leq H_{\e,1}(\theta r;v)\leq C_1\theta H_{\e,1}(r/2;v),
\]
where $C_1$ a constant depending on $d$, $\kappa_1$, $\kappa_2$.  By Lemma~\ref{fiftytwo}, we have the desired estimate.
\end{proof}

\begin{lemm}\label{sixtyfive}
Suppose $\mathcal{L}_{\e,\d}(u_{\e,\d})=0$ in $B(1)$.  For any $\e\leq r\leq 1/3$,
\[
H_{\e,\d}(\theta r;u)\leq C_1\theta H_{\e,\d}(r;u)+\dfrac{C_2}{r}\left(\dfrac{\e}{r}\right)^{\mu}\underset{q\in\R^d}{\inf}\left(\dashint_{B(r)}|\lr{}(u-q)|^2\right)^{1/2}
\]
where $u\equiv u_{\e,\d}$, $\theta\in (0,1/4)$, and $\mu>0$.
\end{lemm}

\begin{proof}
Fix $r\geq \e$, and let $v\equiv v_r$ denote the function given by Lemma~\ref{fiftyone}.  We have
\begin{align*}
H_{\e,\d}(\theta r;u) &\leq \dfrac{1}{\theta r}\left(\dashint_{B(\theta r)}|\lr{}(u_{}-v)|^2\right)^{1/2}+H_{\e,\d}(\theta r;v) \\
&\leq\dfrac{C}{r}\left(\dashint_{B(r)}|\lr{}(u-v)|^2\right)^{1/2}+C_1\theta H_{\e,\d}(r;v) \\
&\leq\dfrac{C}{r}\left(\dashint_{B(r)}|\lr{}(u-v)|^2\right)^{1/2}+C_1\theta H_{\e,\d}(r;u),
\end{align*}
where we've used Lemma~\ref{fiftyfour}.  By Lemma~\ref{fiftyone},
\begin{equation}\label{fiftysix}
H_{\e,\d}(\theta r; u)\leq \dfrac{C_2}{r}\left(\dfrac{\e}{r}\right)^{\mu}\left(\dashint_{B(3r)}|\lr{}u_{}|^2\right)^{1/2}+C_1\theta H_{\e,\d}(r;u).
\end{equation}
Since~\eqref{fiftysix} remains invariant if we subtract a constant from $u$, the desired estimate follows.
\end{proof}

\begin{lemm}\label{sixtythree}
Let $H(r)$ and $h(r)$ be two nonnegative continous functions on the interval $(0,1]$.  Let $0<\e<1/6$.  Suppose that there exists a constant $C_0$ with
\begin{equation}\label{fiftyseven}
\begin{cases}
\underset{r\leq t\leq 3r}{\max} H(t)\leq C_0 H(3r), \\
\underset{r\leq t,s\leq 3r}{\max} |h(t)-h(s)|\leq C_0H(3r), \\
\end{cases}
\end{equation}
for any $r\in [\e,1/3]$.  We further assume
\begin{equation}\label{fiftyeight}
H(\theta r)\leq \dfrac{1}{2}H(r)+C_0\left(\dfrac{\e}{r}\right)^{\mu}\left\{H(3r)+h(3r)\right\}
\end{equation}
for any $r\in [\e,1/3]$ and some $\mu>0$, where $\theta\in (0,1/4)$.  Then 
\[
\underset{\e\leq r\leq 1}{\max}\left\{H(r)+h(r)\right\}\leq C\{H(1)+h(1)\},
\]
where $C$ depends on $C_0$ and $\theta$.
\end{lemm}

\begin{proof}
See~\cite[Lemma 8.5]{shen}.
\end{proof}

\begin{proof}[Proof of Theorem~\ref{five}]
By rescaling, we may assume $R=1$.  We assume $\e\in(0,1/6)$, and we let $H(r)\equiv H_{\e,\d}(r;u)$, where $u\equiv u_{\e,\d}$ and $H_{\e,\d}(r;u)$ is defined above by~\eqref{fortyfive}.  Let $h(r)=r^{-1}|M_r|$, where $M_r\in\R^{d\times d}$ satisfies
\[
H(r)=\dfrac{1}{r}\underset{q\in\R^d}{\inf}\left(\dashint_{B(r)}|\lr{}(u-M_rx-q)|^2\right)^{1/2}.
\]
Note there exists a constant $C$ independent of $r$ so that
\begin{equation}\label{sixtyone}
H(t)\leq C H(3r),\,\,\,t\in [r,3r].
\end{equation}
Suppose $s,t\in [r,3r]$.  We have
\begin{align}
|h(t)-h(s)| &\leq \dfrac{C}{r}\underset{q\in\R^d}{\inf}\left(\dashint_{B(r)}\lr{}|(M_t-M_s)x-q|^2\right)^{1/2} \nonumber\\
&\leq \dfrac{C}{t}\underset{q\in\R^d}{\inf}\left(\dashint_{B(t)}\lr{}|u-M_tx-q|^2\right)^{1/2} \nonumber\\
&\hspace{10mm}+\dfrac{C}{s}\underset{q\in\R^d}{\inf}\left(\dashint_{B(s)}\lr{}|u-M_sx-q|^2\right)^{1/2} \nonumber\\
&\leq C H(3r),\nonumber
\end{align}
where we've used~\eqref{sixtyone} for the last inequality.  Specifically,
\begin{equation}\label{sixtytwo}
\underset{r\leq t,s\leq 3r}{\max}|h(t)-h(s)|\leq CH(3r).
\end{equation}
Clearly
\[
\dfrac{1}{r}\underset{q\in\R^d}{\inf}\left(\dashint_{B(3r)}|\lr{}(u-q)|^2\right)^{1/2}\leq H(3r)+h(3r),
\]
and so Lemma~\ref{sixtyfive} implies
\begin{equation}\label{sixtyfour}
H(\theta r)\leq \dfrac{1}{2}H(r)+C\left(\dfrac{\e}{r}\right)^{\mu}\left\{H(3r)+h(3r)\right\}
\end{equation}
for any $r\in [\e,1/3]$ and some $\theta\in (0,1/4)$.  Note equations~\eqref{sixtyone},~\eqref{sixtytwo}, and~\eqref{sixtyfour} show that $H(r)$ and $h(r)$ satisfy the assumptions of Lemma~\ref{sixtythree}.  Consequently, 
\begin{align}
\left(\dashint_{B(r)}|\lr{}\nabla u|^2\right)^{1/2}&\leq \dfrac{C}{r}\underset{q\in\R^d}{\inf}\left(\dashint_{B(3r)}|\lr{}(u-q)|^2\right)^{1/2} \nonumber\\
&\leq C\left\{H(3r)+h(3r)\right\} \nonumber\\
&\leq C\left\{H(1)+h(1)\right\} \nonumber\\
&\leq C\left(\dashint_{B(1)}|\lr{}u|^2\right)^{1/2}.\label{sixtysix}
\end{align}
Since~\eqref{sixtysix} remains invariant if we subtract a constant from $u$, the desired estimate in Theorem~\ref{five} follows from Poincar\'{e}'s inequality.
\end{proof}

%
%
%
%
%
%
%
\appendix
\section{Interior estimates at the small-scale}\label{section5}
%
%
%
%
%
%

In this appendix, we discuss combining the large-scale estimate Theorem~\ref{five} with $C^{1,\a}$ estimates for interface problems to derive interior estimates at both the macroscopic and microscopic scale.  In particular, we show Corollary~\ref{threethousand}.  First, we prove the following lemma.  

\begin{lemm}\label{seven}
Suppose $A$ satisfies~\eqref{one},~\eqref{two},~\eqref{four}, and is $\a$-H\"older continuous for some $\a\in (0,1)$, i.e., $A$ satisfies~\eqref{six}.  Suppose $\omega$ is an unbounded $C^{1,\a}$ domain.  Let $u_{1,\d}$ denote a weak solution to $\mathcal{L}_{1,\d}(u_\d)=0$ in $B(x_0,1)$ for some $x_0\in\R^d$.  Then
\[
\|\nabla u_\d\|_{C^{0,\a}(B(x_0,r)\cap\e\omega)}+\d\|\nabla u_\d\|_{C^{0,\a}(B(x_0,r)\backslash\e\omega)}\leq C\|\lrnoe{}\nabla u_\d\|_{L^2(B(x_0,1))},
\]
for a constant $C$ independent of $\d$ and $0<r\leq 1/3$.  In particular,
\[
\|\lrnoe{}\nabla u_\d\|_{L^\infty(B(x_0,r))}\leq C\|\lrnoe{}\nabla u_\d\|_{L^2(B(x_0,1))}
\]
for $0<r\leq 1/3$.
\end{lemm}

Lemma~\ref{seven} was proved for scalar equations with diagonal coeffcients in smooth domains in~\cite{fabes,yeh10,yeh16}.  Lemma~\ref{seven} continues to hold for elliptic systems with coefficients and domains satisfying the given assumptions.  Together, Lemma~\ref{seven} and Theorem~\ref{five} give interior Lipschitz estimates for $\mathcal{L}_{\e,\d}$ at every scale.

Let $\G(\cdot,x)$ denote the matrix-valued fundamental solution associated with $\mathcal{L}_{1,1}$ in $\R^d$.  That is, $\Gamma(\cdot,x)=\{\Gamma^{\a\b}(\cdot,x)\}_{1\leq\a,\b\leq d}$ satisfies
\[
f^\b(x)=\dint_{\R^d}a_{ij}^{\a\g}(\x)\dfrac{\dd\Gamma^{\g\b}}{\dd x_j}(\x,x)\dfrac{\dd f^\a}{\dd x_i}(\x)d\sigma(\x)
\]
for $f=\{f^\b\}_{1\leq\b\leq d}\in C_0^\infty (\R^d;\R^d)$.  Indeed, if $A$ is VMO, i.e., 
\begin{equation}\label{sixtyseven}
\underset{\substack{x\in\R^d \\ 0<r<R}}{\sup}\dashint_{B(x,r)}\left|A(y)-\dashint_{B(x,r)}A\right|\,dy\to 0\,\,\,\text{ as }R\to 0^+
\end{equation}
then $\G(\cdot,x)\in W^{1,1}_{\text{loc}}(\R^d\backslash\{x\};\R^{d\times d})$ exists uniquely for each $x\in\R^d$ (see work of Hofmann and Kim~\cite{hofmann} for $d\geq 3$ and work of Brown, Kim, and Taylor~\cite{brownie} for $d=2$).  If $A$ satisfies~\eqref{six}, then $A$ satisfies~\eqref{sixtyseven}.  

For a bounded, simply-connected domain $H$ and $g\in L^2(\dd H;\R^d)$, the single-layer potential $\S g=\{(\S g)^\a\}_{1\leq \a\leq d}$ is given by
\begin{align}
(\S g)^\a(x)=\dint_{\dd H} \Gamma^{\a\b}(x,\x)g^\b(\x)\,d\sigma(\x),\,\,\,x\in\R^d\backslash\dd H \label{seventyseven}
\end{align}
and the double-layer potential $\D g=\{(\D g)^\a\}_{1\leq \a\leq d}$ is given by
\begin{equation}\label{seventyeight}
(\D g)^\a(x)=\dint_{\dd H}n_i(\x)a_{ij}^{\a\b}(\x)\dfrac{\dd\Gamma^{\b\g}}{\dd x_j}(\x,x)g^\g(\x)\,d\sigma(\x),\,\,\,x\in\R^d\backslash\dd H
\end{equation}
where $n(\x)=\{n_i(\x)\}_{1\leq i\leq d}$ denotes the unit vector outward normal to $H$ at $\x\in\dd H$.

It is known (see~\cite[Theorem 4.6]{potential}) that if $g\in L^2(\dd H;\R^d)$, then
\begin{equation}\label{seventynine}
\D g^{\pm}=\pm\dfrac{1}{2}g+\mathcal{K}g\,\,\,\text{ on }\dd H,
\end{equation}
where $\K$ is given by
\[
\K g(x)=\text{p.v.}\dint_{\dd H}n_i(\x)a_{ij}^{\a\b}(\x)\dfrac{\dd \Gamma^{\b\g}}{\dd x_j}(\x,x)g^\g(\x)\,d\sigma(\x),\,\,\,x\in\dd H
\]
and
\[
\D g^{\pm}(x)=\lim_{h\to 0^+}\D g(x\pm hn),
\]
i.e., $\D g^+$ and $\D g^-$ denote the traces of $\D g$ on $\dd H$ from the exterior of $\overline{H}$ and the interior of $H$, respectively.  In particular, $w=\D g$ satisfies $\mathcal{L}_{1,1}(w)=0$ in $H$ and $w=(-\frac{1}{2}+\K)g$ on $\dd H$.  It is also known (see~\cite[Lemma 5.7]{potential}) that if $H$ is Lipschitz and $A$ satisfies~\eqref{one},~\eqref{two}, and~\eqref{six}, then
\begin{equation}\label{eighty}
-\dfrac{1}{2}+\mathcal{K}:L^2(\dd H;\R^d)\to L^2(\dd H;\R^d)
\end{equation}
is bounded and continuously invertible.  For single equations, this follows from the compactness of $\mathcal{K}$ and Fredholm theory (see the argument of Yeh in~\cite[Lemma 3.2]{yeh10}).  For systems with variable coefficients, the operator $\mathcal{K}$ is not compact (see the work of Kenig and Shen~\cite{potential} for an alternative proof of invertibility on $L^2$).  The following lemma, however, is more or less known.

\begin{lemm}\label{eightyone}
Suppose $A$ satisfies~\eqref{one},~\eqref{two},~\eqref{four} and is $\a$-H\"older continous, i.e., satisfies~\eqref{six}, for some $\a\in (0,1)$.  Suppose $H$ is a bounded $C^{1,\a}$ domain.  The operators
\[
\mathcal{S}:C^{0,\a}(\dd H;\R^d)\mapsto C^{1,\a}(\dd H;\R^d)
\]
and
\[
\D:C^{1,\a}(\dd H;\R^d)\mapsto C^{1,\a}(\dd H;\R^d)
\]
defined by~\eqref{seventyseven} and~\eqref{seventyeight}, respectively, are bounded.
\end{lemm}

From~\eqref{seventynine}, we have the jump relations
\begin{equation}\label{eightythree}
g=w^+-w^-\,\,\,\,\,\,\text{ and }\,\,\,\,\,\,\left(\dfrac{\dd w}{\dd n}\right)^+=\left(\dfrac{\dd w}{\dd n}\right)^{-},
\end{equation}
where $w=\D g$ and $\dd w/\dd n=n\cdot\nabla w$ denotes the normal derivative of $w$.

The following lemma essentially follows from the jump relations~\eqref{eightythree} and regularity problems for the exterior Neumann and interior Dirichlet problems.

\begin{lemm}\label{eightytwo}
There exists a constant $C$ depending on $\kappa_1$, $\kappa_2$ and $H$ such that
\[
\|g\|_{1,\a}\leq C\left\|\left(-\dfrac{1}{2}+\K\right)g\right\|_{1,\a}
\]
for any $g\in C^{1,\a}(\dd H;\R^d)$, where $\|\cdot\|_{1,\a}=\|\cdot\|_{C^{1,\a}(\dd H)}$.
\end{lemm}

%

As mentioned in Section~\ref{section2}, any two connected components of $\R^d\backslash\omega$ are separated by some positive distance $\mathfrak{g}^\omega$.  If $\R^d\backslash\omega=\cup_{k=1}^\infty H_k$, write $H_k^*$ to denote the set
\[
H_k^*=\{x\in\R^d\,:\,\text{dist}(x,H_k)\leq \mathfrak{g}^\omega/4\}.
\]
To prove Lemma~\ref{seven}, it suffices to show the result holds in each $H^*_k$.  Indeed, if $A$ satisfies~\eqref{six}, the boundedness of $\nabla u_{1,\d}$ in the interior of $\omega$ follows from classical results regarding elliptic systems with H\"{o}lder continuous coefficients.

\begin{lemm}
Suppose $A$ satisfies~\eqref{one},~\eqref{two}, and is $\a$-H\"older continuous for some $\a\in (0,1)$.  Suppose $\omega$ is an unbounded $C^{1,\a}$ domain.  If $\mathcal{L}_{1,\d}(u_{1,\d})=\text{div}(f)$ in $H_k^*$ and $u_{1,\d}=0$ on $\dd H_k^*$, then
\[
\|\lrnoe{}\nabla u_{1,\d}\|_{C^{0,\a}(H_k^*)}\leq C\left\{\|\lrnoe{}\nabla u_{1,\d}\|_{L^2(H_k^*)}+\|\lrnoe{-1}f\|_{C^{0,\a}(H_k^*)}\right\},
\]
where $C$ depends only on $\|A\|_{C^{\a}}$, $\a$, $\omega$, $\kappa_1$, $\kappa_2$.
\end{lemm}

\begin{proof}
Note that if $\d_0\leq \d\leq 1$, then the result follows from general results regarding divergence form elliptic equations with $\a$-H\"older continuous coefficients in $C^{1,\a}$ domains.  Hence, we may assume $0\leq \d\leq\d_0$ for some $\d_0$ to be determined.

Let $u_1$ satisfy the boundary value problem
\[
\begin{cases}
-\text{div}(\d^2A\nabla u_1)=\text{div}(f)\,\,\,\text{ in }H_k \\
-\text{div}(A\nabla u_1)=\text{div}(f)\,\,\,\text{ in }H_k^*\backslash H_k \\
u_1=0\,\,\,\text{ on }\dd H_k\cup \dd H_k^*
\end{cases}
\]
By $C^{1,\a}$ estimates for elliptic systems with $\a$-H\"{o}lder continuous coefficients in $C^{1,\a}$ domains (see~\cite[Chapter 9, Theorem 2.7]{chenwu}), we have
\begin{equation}\label{seventyfour}
\|\lrnoe{}\nabla u_1\|_{C^{0,\a}(H_k^*)}\leq C\|\lrnoe{-1}f\|_{C^{0,\a}(H_k^*)}.
\end{equation}
Set $u_2=u-u_1$, where $u\equiv u_{1,\d}$.  Note then $u_2$ satisfies the equation and jump conditions
\begin{equation}\label{seventytwo}
\begin{cases}
-\text{div}(\lrnoe{2}A\nabla u_2)=0\,\,\,\text{ in } H_k^* \\
\lfloor \lrnoe{2}A\nabla u_2\rfloor_{\dd H_k}\cdot n=-\lfloor \lrnoe{2}A\nabla u_1\rfloor_{\dd H_k}\cdot n \\
\lfloor u_2\rfloor_{\dd H_k}=0, \\
u_2=0\,\,\,\text{ on }\dd H_k^*
\end{cases}
\end{equation}
where $\lfloor g\rfloor_{\dd H_k}=g^+-g^-$, $g^{\pm}=\lim_{t\to 0^+}g(\cdot\pm tn)$, and $n$ denotes the unit vector outward normal to $H_k$.  Hence, for $x\in H_k$,
\begin{equation}\label{sixtyeight}
u_2(x)=-\dint_{\dd H_k}\dfrac{\dd\Gamma}{\dd n_A}(x,y)u(y)\,d\sigma(y)+\Gamma(x,y)\dfrac{\dd u_2}{\dd n_A}(y)\,d\sigma(y),
\end{equation}
where $\dd g/\dd n_A=A\nabla g\cdot n$.  For $x\in H_k^*\backslash H_k$,
\begin{align}
u_2(x)&=
\dint_{\dd H^*_k}\Gamma(x,y)\dfrac{\dd u_2}{\dd n^*_A}(y)\,d\sigma(y) \nonumber\\
&\hspace{5mm}-\dint_{\dd H_k}\Gamma(x,y)\dfrac{\dd u_2}{\dd n_A}(y)\,d\sigma(y)+\dfrac{\dd\Gamma}{\dd n_A}(x,y)u(y)\,d\sigma(y),\label{sixtynine}
\end{align}
where $n^*$ denotes the unit vector outward normal to $H_k^*$.  Then~\eqref{sixtyeight} and~\eqref{sixtynine} imply
\begin{equation}\label{seventy}
u(x)=\left\{\dfrac{1}{2}u(x)-\mathcal{D}u(x)\right\}+\dint_{\dd H_k}\Gamma(x,y)\dfrac{\dd u^-_2}{\dd n_A}(y)\,d\sigma(y),
\end{equation}
and
\begin{align}
u(x)&=
\dint_{\dd H^*_k}\Gamma(x,y)\dfrac{\dd u_2}{\dd n^*_A}(y)\,d\sigma(y) \nonumber\\
&\hspace{5mm}-\dint_{\dd H_k}\Gamma(x,y)\dfrac{\dd u^+_2}{\dd n_A}(y)\,d\sigma(y)-\left\{-\dfrac{1}{2}u(x)-\mathcal{D}u(x)\right\}\label{seventyone},
\end{align}
for $x\in \dd H_k$ (see~\cite[Chapter 7]{mclean}), where $\mathcal{D}\equiv\mathcal{D}_{\dd H_k}$.  Equations~\eqref{seventytwo},~\eqref{seventy}, and~\eqref{seventyone} then imply
\begin{align*}
&\left\{\dfrac{1}{2}+\mathcal{D}\right\}u(x)=\mathcal{S}\left(\dfrac{\dd u^-_2}{\dd n_A}\right)(x) \\
&\left\{\dfrac{1}{2}-\mathcal{D}\right\}u(x)=\dint_{\dd H_k^*}\Gamma(x,y)\dfrac{\dd u_2}{\dd n^*_A}(y)\,d\sigma(y)-\mathcal{S}\left(\dfrac{\dd u^+_2}{\dd n_A}\right)(x),
\end{align*}
where $\mathcal{S}\equiv \mathcal{S}_{\dd H_k}$.  Finally, by~\eqref{seventytwo},
\begin{align*}
&\left[1-2\left(\dfrac{1-\d^2}{1+\d^2}\right)\mathcal{D}\right]u(x) \\
&\hspace{10mm}=\dfrac{2}{1+\d^2}\left\{\dint_{\dd H_k^*}\Gamma(x,y)\dfrac{\dd u_2}{\dd n^*_A}(y)\,d\sigma(y)+\mathcal{S}\left(\lfloor\lrnoe{2}A\nabla u_1\rfloor_{\dd H_k}\cdot n\right)\right\}.
\end{align*}
Choose $\d_0$ small enough so that by Lemma~\ref{eightytwo} we have
\begin{align*}
\|\nabla u_2\|_{C^{0,\a}(\dd H_k)}\leq C\left\{\|\lfloor\lrnoe{2}A\nabla u_1\rfloor_{\dd H_k}\|_{C^{0,\a}(\dd H_k)}+\|\dd u_2/\dd n_A\|_{C^{0,\a}(\dd H^*_k)}\right\},
\end{align*}
for $0\leq \d\leq \d_0$, where $C$ is some constant independent of $\d$.  Indeed, it is sufficient to take $\d_0$ so that
\[
4C\left(\dfrac{\d^2}{1+\d^2}\right)<1\,\,\,\text{ for }\d\leq \d_0,
\]
where $C$ is a constant depending only on the operator norm of $\D$, which is finite by Lemma~\eqref{eightyone}.  By~\eqref{seventytwo} and~\eqref{seventyfour},
\begin{equation}\label{seventyfive}
\|\lrnoe{}\nabla u_2\|_{C^{0,\a}(H_k^*)}\leq C\|\lrnoe{-1}f\|_{C^{0,\a}(H_k^*)}.
\end{equation}
Equations~\eqref{seventyfour} and~\eqref{seventyfive} give the desired estimate.
\end{proof}

\begin{proof}[Proof of Corollary~\ref{threethousand}]
By rescaling, we may assume $R=1$.  To prove the desired estimate, assume $\e\in (0,1/9)$.  Indeed, if $\e\geq 1/9$, then~\eqref{eight} follows from Theorem~\ref{five}.  From Lemma~\ref{seven}, Theorem~\ref{five}, and a ``blow-up argument'' (see the proof of Lemma~\ref{fiftyone}), we deduce
\begin{align*}
\|\lr{}\nabla u_{\e,\d}\|_{L^\infty(B(y,\e))} &\leq C\left(\dashint_{B(y,3\e)}|\lr{}\nabla u_{\e,\d}|^2\right)^{1/2} \\
&\leq C\left(\dashint_{B(x_0,1)}|\lr{}\nabla u_{\e,\d}|^2\right)^{1/2}
\end{align*}
for any $y\in B(x_0,1/3)$.  The desired estimate follows by covering $B(x_0,1/3)$ with balls $B(y,\e)$.
\end{proof}

\bibliography{/Users/bchaserussell/Documents/LaTeX/Research/Bibliography/masterbib}{}
\bibliographystyle{plain}

\vspace{5mm}

\text{Brandon Chase Russell}

\textsc{Department of Mathematics}

\textsc{University of Kentucky}

\textsc{Lexington, KY 40506, USA}

\text{E-mail:} brandon.russell700@uky.edu

\end{document}